\documentclass{amsart}%
\usepackage{amsfonts}
\usepackage{amsmath}
\usepackage{amssymb}
\usepackage{graphicx}
\usepackage[latin1]{inputenc}
\usepackage[english]{babel}%
\newtheorem{theorem}{Theorem}[section]
\theoremstyle{plain}
\newtheorem{acknowledgement}[theorem]{Acknowledgement}

\newtheorem{corollary}[theorem]{Corollary}

\newtheorem{lemma}[theorem]{Lemma}

\newtheorem{proposition}[theorem]{Proposition}
\newtheorem{remark}[theorem]{Remark}

\numberwithin{equation}{section}
\begin{document}
\title[Non-Archimedean Parabolic-type Equations ]{Non-local Operators, Non-Archimedean Parabolic-type Equations with Variable
Coefficients and Markov Processes}
\author{L. F. Chacón-Cortes}
\author{W. A. Zúñiga-Galindo}
\address{Centro de Investigación y de Estudios Avanzados del Instituto Politécnico Nacional\\
Departamento de Matemáticas- Unidad Querétaro\\
Libramiento Norponiente \#2000, Fracc. Real de Juriquilla. Santiago de
Querétaro, Qro. 76230\\
México}
\email{fchaconc@math.cinvestav.edu.mx, wazuniga@math.cinvestav.edu.mx}
\thanks{The second author was partially supported by Conacyt (Mexico), Grant \# 127794.}

\begin{abstract}
In this article, we introduce a new class of parabolic-type
pseudo\-\-\allowbreak differential equations with variable coefficients over
the $p$-adics. We establish the existence and uniqueness of solutions for the
Cauchy problem associated with these equations. The fundamental solutions of
these equations are connected with Markov processes. Some of these equations
are related to new models of complex systems.

\end{abstract}
\maketitle

\section{Introduction}

Stochastic processes on $p$-adic spaces, or more generally on ultrametric
spaces, have been studied extensively due to its connections with models of
complex systems, see e.g. \cite{A-K1}-\cite{A-K2}, \cite{Av-3}-\cite{Av-5},
\cite{Ch-Z}, \cite{Dra-Kh-K-V}, \cite{Ka}, \cite{Koch}, \cite{K-Kos},
\cite{M-P-V}, \cite{R-T-V}, \cite{V-V-Z}, and the references therein. In
\cite{Av-3}-\cite{Av-5}, Avetisov et al. introduced a new class of models for
complex systems based on $p$-adic analysis, these models can be applied, for
instance, to the study the relaxation of biological complex systems. From a
mathematical point view, in these models the time-evolution of a complex
system is described by a $p$-adic master equation (a parabolic-type
pseudodifferential equation) which controls the time-evolution of a transition
function of a random walk on an ultrametric space, and the random walk
describes the dynamics of the system in the space of configurational states
which is approximated by an ultrametric space ($\mathbb{Q}_{p}$). The simplest
type of master equation is the one-dimensional $p$-adic heat equation. This
equation was introduced in the book of Vladimirov, Volovich and Zelenov
\cite[Section \ XVI]{V-V-Z}. In \cite[Chapters 4, 5]{Koch} Kochubei presented
a general theory for one-dimensional parabolic-type pseudodifferential
equations with variable coefficients, whose fundamental solutions are
transition density functions for Markov processes in the $p$-adic line, see
also \cite{Ro}, \cite{R-Zu}, \cite{Va1}. In \cite{Zu}, the second author
introduced $p$-adic analogs for the $n$-dimensional elliptic operators and
studied the corresponding heat equations and the associated Markov processes,
see also  \cite{C-Z2}, \cite{Va1}. 

In \cite{Ch-Z}, the authors introduced a new type of non-local operators which
are na\-turally connected with \ parabolic-type pseudodifferential equations.
Building up on \cite{Ch-Z} and \cite{Koch}-\cite{Koch0}, in this arti\-cle, we
introduce a new class of parabolic-type pseudodifferential equations with
variable coefficients, which contains the one-dimensional $p$-adic heat
equation of \cite{V-V-Z}, the equations studied by Kochubei in \cite{Koch},
and the equations studied by Rodríguez-Vega in \cite{Ro}. Our theory is not
applicable to the equations studied in \cite{Zu}, \cite{C-Z2}. We establish
the existence and uniqueness of solutions for the Cauchy problem for such
equations, see Theorems \ref{Thm1}, \ref{mainTheo}, \ref{uniqueTheo}. We show
that the fundamental solutions of these equations are transition density
functions of Markov processes, see Theorem \ref{Thm2}. Finally, we study the
well-possednes of the Cauchy problem, see Theorem \ref{Thm3}.

\section{\label{Section1}Preliminaries}

In this section we fix the notation and collect some basic results on $p$-adic
analysis that we will use through the article. For a detailed exposition the
reader may consult \cite{A-K-S}, \cite{Taibleson}, \cite{V-V-Z}.

\subsection{The field of $p$-adic numbers}

Along this article $p$ will denote a prime number. The field of $p-$adic
numbers $\mathbb{Q}_{p}$ is defined as the completion of the field of rational
numbers $\mathbb{Q}$ with respect to the $p-$adic norm $|\cdot|_{p}$, which is
defined as
\[
|x|_{p}=%
\begin{cases}
0 & \text{if }x=0\\
p^{-\gamma} & \text{if }x=p^{\gamma}\dfrac{a}{b},
\end{cases}
\]
where $a$ and $b$ are integers coprime with $p$. The integer $\gamma:=ord(x)$,
with $ord(0):=+\infty$, is called the\textit{ }$p-$\textit{adic order of} $x$.
We extend the $p-$adic norm to $\mathbb{Q}_{p}^{n}$ by taking%
\[
||x||_{p}:=\max_{1\leq i\leq n}|x_{i}|_{p},\qquad\text{for }x=(x_{1}%
,\dots,x_{n})\in\mathbb{Q}_{p}^{n}.
\]
We define $ord(x)=\min_{1\leq i\leq n}\{ord(x_{i})\}$, then $||x||_{p}%
=p^{-\text{ord}(x)}$. The set $\left(  \mathbb{Q}_{p}^{n},||\cdot
||_{p}\right)  $ is a complete ultrametric space. As a topological space
$\mathbb{Q}_{p}$\ is homeomorphic to a Cantor-like subset of the real line.

Any $p-$adic number $x\neq0$ has a unique expansion $x=p^{ord(x)}\sum
_{j=0}^{\infty}x_{j}p^{j}$, where $x_{j}\in\{0,1,2,\dots,p-1\}$ and $x_{0}%
\neq0$. By using this expansion, we define \textit{the fractional part of
}$x\in\mathbb{Q}_{p}$, denoted $\{x\}_{p}$, as the rational number
\[
\{x\}_{p}=%
\begin{cases}
0 & \text{if }x=0\text{ or }ord(x)\geq0\\
p^{\text{ord}(x)}\sum_{j=0}^{-ord(x)-1}x_{j}p^{j} & \text{if }ord(x)<0.
\end{cases}
\]
For $\gamma\in\mathbb{Z}$, denote by $B_{\gamma}^{n}(a)=\{x\in\mathbb{Q}%
_{p}^{n}:||x-a||_{p}\leq p^{\gamma}\}$ \textit{the ball of radius }$p^{\gamma
}$ \textit{with center at} $a=(a_{1},\dots,a_{n})\in\mathbb{Q}_{p}^{n}$, and
take $B_{\gamma}^{n}(0):=B_{\gamma}^{n}$. Notice that $B_{\gamma}%
^{n}(a)=B_{\gamma}(a_{1})\times\cdots\times B_{\gamma}(a_{n})$, where
$B_{\gamma}(a_{i}):=\{x\in\mathbb{Q}_{p}:|x-a_{i}|_{p}\leq p^{\gamma}\}$ is
the one-dimensional ball of radius $p^{\gamma}$ with center at $a_{i}%
\in\mathbb{Q}_{p}$. The ball $B_{0}^{n}$ equals the product of $n$ copies of
$B_{0}:=\mathbb{Z}_{p}$, \textit{the ring of }$p-$\textit{adic integers}. We
denote by $\Omega(\left\Vert x\right\Vert _{p})$ the characteristic function
of $B_{0}^{n}$. For more general sets, say Borel sets, we use ${\LARGE 1}%
_{A}\left(  x\right)  $ to denote the characteristic function of $A$.

\subsection{The Bruhat-Schwartz space}

A complex-valued function $\varphi$ defined on $\mathbb{Q}_{p}^{n}$ is
\textit{called locally constant} if for any $x\in\mathbb{Q}_{p}^{n}$ there
exists an integer $l=l(x)\in\mathbb{Z}$ such that%
\begin{equation}
\varphi(x+x^{\prime})=\varphi(x)\text{ for }x^{\prime}\in B_{l}^{n}.
\label{local_constancy}%
\end{equation}

The set of all locally constant functions $\varphi$, for which the integer
$l(x)$ is independent of $x$, form $\mathbb{C}$-vector space denoted by
$\widetilde{\mathcal{E}}(\mathbb{Q}_{p}^{n}):=\widetilde{\mathcal{E}}$. Given
$\varphi\in\widetilde{\mathcal{E}}$, we call the largest possible $l=l\left(
\varphi\right)  $, the \textit{parameter of local constancy of} $\varphi$.

A function $\varphi:\mathbb{Q}_{p}^{n}\rightarrow\mathbb{C}$ is called a
\textit{Bruhat-Schwartz function (or a test function)} if it is locally
constant with compact support. The $\mathbb{C}$-vector space of
Bruhat-Schwartz functions is denoted by $S(\mathbb{Q}_{p}^{n}):=S$. Notice
that $S\subset\widetilde{\mathcal{E}}$.

Let $S^{\prime}(\mathbb{Q}_{p}^{n}):=S^{\prime}$ denote the set of all
functionals (distributions) on $S(\mathbb{Q}_{p}^{n})$. All functionals on
$S(\mathbb{Q}_{p}^{n})$ are continuous.

Set $\chi_{p}(y)=\exp(2\pi i\{y\}_{p})$ for $y\in\mathbb{Q}_{p}$. The map
$\chi_{p}(\cdot)$ is an additive character on $\mathbb{Q}_{p}$, i.e. a
continuos map from $\mathbb{Q}_{p}$ into the unit circle satisfying $\chi
_{p}(y_{0}+y_{1})=\chi_{p}(y_{0})\chi_{p}(y_{1})$, $y_{0},y_{1}\in
\mathbb{Q}_{p}$.

Given $\xi=(\xi_{1},\dots,\xi_{n})$ and $x=(x_{1},\dots,x_{n})\in
\mathbb{Q}_{p}^{n}$, we set $\xi\cdot x:=\sum_{j=1}^{n}\xi_{j}x_{j}$. The
Fourier transform of $\varphi\in S(\mathbb{Q}_{p}^{n})$ is defined as
\[
(\mathcal{F}\varphi)(\xi)=%
{\displaystyle\int\limits_{\mathbb{Q}_{p}^{n}}}
\Psi(-\xi\cdot x)\varphi(x)d^{n}x\quad\text{for }\xi\in\mathbb{Q}_{p}^{n},
\]
where $d^{n}x$ is the Haar measure on $\mathbb{Q}_{p}^{n}$ normalized by the
condition $vol(B_{0}^{n})=1$. The Fourier transform is a linear isomorphism
from $S(\mathbb{Q}_{p}^{n})$ onto itself satisfying $(\mathcal{F}%
(\mathcal{F}\varphi))(\xi)=\varphi(-\xi)$. We will also use the notation
$\mathcal{F}_{x\rightarrow\xi}\varphi$ and $\widehat{\varphi}$\ for the
Fourier transform of $\varphi$.

\subsubsection{Fourier transform}

The Fourier transform $\mathcal{F}\left[  T\right]  $ of a distribution $T\in
S^{\prime}\left(  \mathbb{Q}_{p}^{n}\right)  $ is defined by%
\[
\left(  \mathcal{F}\left[  T\right]  ,\varphi\right)  =\left(  T,\mathcal{F}%
\left[  \varphi\right]  \right)  \text{ for all }\varphi\in S\left(
\mathbb{Q}_{p}^{n}\right)  \text{.}%
\]
The Fourier transform $f\rightarrow\mathcal{F}\left[  T\right]  $ is a linear
isomorphism from $S^{\prime}\left(  \mathbb{Q}_{p}^{n}\right)  $\ onto
$S^{\prime}\left(  \mathbb{Q}_{p}^{n}\right)  $. Furthermore, $T=\mathcal{F}%
\left[  \mathcal{F}\left[  T\right]  \left(  -\xi\right)  \right]  $.

\section{\label{Sect2}A class of non-local operators}

Denote by $\mathcal{\mathfrak{M}}_{\lambda}$, with $\lambda\geq0$,\ the
$\mathbb{C}$-vector space of all the functions $\varphi\in\widetilde
{\mathcal{E}}$ satisfying $\left\vert \varphi(x)\right\vert \leq
C(1+\left\Vert x\right\Vert _{p}^{\lambda})$. If the function $\varphi$
depends also on a parameter $t$, we shall say that \ $\varphi$ belongs to
$\mathcal{\mathfrak{M}}_{\lambda}$ \textit{uniformly with respect to} $t$, if
its constant $C$ and\ its parameter of local constancy do \ not depend on $t$.
Notice that, if$\ 0\leq\lambda_{1}\leq\lambda_{2}$, then
$\mathcal{\mathfrak{M}}_{0}\subseteq$\ $\mathcal{\mathfrak{M}}_{\lambda_{1}%
}\subseteq$ $\mathcal{\mathfrak{M}}_{\lambda_{2}},$ and that $S(\mathbb{Q}%
_{p}^{n})\subseteq$ $\mathcal{\mathfrak{M}}_{0}$.

Take $\mathbb{R}_{+}:=\left\{  x\in\mathbb{R};x\geq0\right\}  $, and fix a
function%
\[
w_{\alpha}:\mathbb{Q}_{p}^{n}\rightarrow\mathbb{R}_{+}%
\]
having the following properties:

\noindent(i) $w_{\alpha}\left(  y\right)  $ is a radial (i.e. $w_{\alpha
}\left(  y\right)  =w_{\alpha}\left(  \left\Vert y\right\Vert _{p}\right)
$)\ and continuous function;

\noindent(ii) $w_{\alpha}\left(  y\right)  =0$ if and only if $y=0$;

\noindent(iii) there exist constants $C_{0},C_{1}>0$, and $\alpha>n$ such
that
\[
C_{0}\left\Vert y\right\Vert _{p}^{\alpha}\leq w_{\alpha}(\left\Vert
y\right\Vert _{p})\leq C_{1}\left\Vert y\right\Vert _{p}^{\alpha}\text{ for
any }y\in\mathbb{Q}_{p}^{n}.
\]

Set
\[
A_{w_{\alpha}}\left(  \xi\right)  :={\int\limits_{\mathbb{Q}_{p}^{n}}}%
\frac{1-\Psi\left(  -y\cdot\xi\right)  }{w_{\alpha}\left(  \left\Vert
y\right\Vert _{p}\right)  }d^{n}y.
\]

In \ \cite{Ch-Z},\ we establish that function $A_{w_{\alpha}}$ is radial,
positive, continuous, $A_{w_{\alpha}}\left(  0\right)  =0$, and $A_{w_{\alpha
}}\left(  \xi\right)  =A_{w_{\alpha}}\left(  \left\Vert \xi\right\Vert
_{p}\right)  =A_{w_{\alpha}}\left(  p^{-ord(\xi)}\right)  $ is a decreasing
function of $ord(\xi)$, cf. \cite[Lemma 3.2]{Ch-Z}. In addition, we introduce
the following operator:%
\begin{equation}
(\boldsymbol{W}_{\alpha}\varphi)(x)={\int\limits_{\mathbb{Q}_{p}^{n}}}%
\frac{\varphi\left(  x-y\right)  -\varphi\left(  x\right)  }{w_{\alpha}\left(
\left\Vert y\right\Vert _{p}\right)  }d^{n}y\text{, }\varphi\in S(\mathbb{Q}%
_{p}^{n})\text{.} \label{W}%
\end{equation}
\ 

\begin{lemma}
\label{lemma1}If $\alpha-n>\lambda$, then $\boldsymbol{W}_{\alpha}$ can be
extended to $\mathcal{\mathfrak{M}}_{\lambda}$ and formula (\ref{W}) holds.
Furthermore, \ $\boldsymbol{W}_{\alpha}:\mathcal{\mathfrak{M}}_{\lambda
}\rightarrow\mathcal{\mathfrak{M}}_{\lambda}$.
\end{lemma}

\begin{proof}
Notice that if $\varphi\in\mathcal{\mathfrak{M}}_{\lambda}$, there exists a
constant $l=l\left(  \varphi\right)  \in\mathbb{Z}$, such that
\begin{equation}
(\boldsymbol{W}_{\alpha}\varphi)(x)={\int\limits_{\left\Vert y\right\Vert
_{p}\geq p^{l}}}\frac{\varphi\left(  x-y\right)  -\varphi\left(  x\right)
}{w_{\alpha}\left(  \left\Vert y\right\Vert _{p}\right)  }d^{n}y. \label{W1}%
\end{equation}
We now show that $\left\vert (\boldsymbol{W}_{\alpha}\varphi)(x)\right\vert
\leq A(1+\left\Vert x\right\Vert _{p}^{\lambda})$. By using that $\varphi
\in\mathcal{\mathfrak{M}}_{\lambda}$, and $\alpha>n$,%
\[
\left\vert (\boldsymbol{W}_{\alpha}\varphi)(x)\right\vert \leq C{\int
\limits_{\left\Vert y\right\Vert _{p}\geq p^{l}}}\frac{(1+\left\Vert
x-y\right\Vert _{p}^{\lambda})}{\left\Vert y\right\Vert _{p}^{\alpha}}%
d^{n}y+C^{\prime}(1+\left\Vert x\right\Vert _{p}^{\lambda}){.}%
\]
Hence, it is sufficient to show that the above integral can be bounded by
$A(1+\left\Vert x\right\Vert _{p}^{\lambda})$, for some positive constant $A$.
If $\left\Vert x\right\Vert _{p}>\left\Vert y\right\Vert _{p}$,
\begin{align*}
{\int\limits_{\left\Vert y\right\Vert _{p}\geq p^{l}}}\frac{(1+\left\Vert
x-y\right\Vert _{p}^{\lambda})}{\left\Vert y\right\Vert _{p}^{\alpha}}d^{n}y
&  \leq(1+\left\Vert x\right\Vert _{p}^{\lambda}){\int\limits_{\left\Vert
y\right\Vert _{p}\geq p^{l}}}\frac{1}{\left\Vert y\right\Vert _{p}^{\alpha}%
}d^{n}y\\
&  =B(1+\left\Vert x\right\Vert _{p}^{\lambda}),
\end{align*}
where $B$ is a positive constant. If $\left\Vert x\right\Vert _{p}<\left\Vert
y\right\Vert _{p}$, by using $\alpha-n>\lambda$,
\[
{\int\limits_{\left\Vert y\right\Vert _{p}\geq p^{l}}}\frac{(1+\left\Vert
x-y\right\Vert _{p}^{\lambda})}{\left\Vert y\right\Vert _{p}^{\alpha}}%
d^{n}y\leq{\int\limits_{\left\Vert y\right\Vert _{p}\geq p^{l}}}%
\frac{(1+\left\Vert y\right\Vert _{p}^{\lambda})}{\left\Vert y\right\Vert
_{p}^{\alpha}}d^{n}y<\infty.
\]
If $\left\Vert x\right\Vert _{p}=\left\Vert y\right\Vert _{p}\geq p^{l}$, we
take $x=p^{L}u$, $y=p^{L}v$, with $\left\Vert v\right\Vert _{p}=\left\Vert
u\right\Vert _{p}=1$, $L\in\mathbb{Z}$, then
\begin{align*}
{\int\limits_{\left\Vert y\right\Vert _{p}=\left\Vert x\right\Vert _{p}}}%
\frac{(1+\left\Vert x-y\right\Vert _{p}^{\lambda})}{\left\Vert y\right\Vert
_{p}^{\alpha}}d^{n}y  &  =p^{-L\left(  n-\alpha\right)  }{\int
\limits_{\left\Vert v\right\Vert _{p}=1}}(1+p^{-L\lambda}\left\Vert
u-v\right\Vert _{p}^{\lambda})d^{n}v\\
&  \leq A\left(  \left\Vert x\right\Vert _{p}^{-\left(  \alpha-n\right)
}+\left\Vert x\right\Vert _{p}^{-\left(  \alpha-n-\lambda\right)  }\right)
\leq A^{\prime}\left(  p,l,\alpha,n,\lambda\right)  ,
\end{align*}
where $A$, $A^{\prime}$ are positive constants.

Finally, by (\ref{W1}) $\boldsymbol{W}_{\alpha}\varphi$\ is locally constant.
\end{proof}

\section{\label{Sect3}Parabolic-type equations with constant coefficients}

Consider the following Cauchy problem:%
\begin{equation}
\left\{
\begin{array}
[c]{ll}%
\frac{\partial u}{\partial t}(x,t)-\kappa\cdot(\boldsymbol{W}_{\alpha
}u)(x,t)=f(x,t), & x\in\mathbb{Q}_{p}^{n},t\in(0,T]\\
& \\
u\left(  x,0\right)  =\varphi(x), &
\end{array}
\right.  \label{Cauchy1}%
\end{equation}
where, $\alpha>n,$ $\kappa$, $T$ are positive constants, $\varphi\in
Dom(\boldsymbol{W}_{\alpha}):=\mathcal{\mathfrak{M}}_{\lambda}$, with
$\alpha-n>\lambda$, $f$ is continuous in $(x,t)$ and belongs to
$\mathcal{\mathfrak{M}}_{\lambda}$ uniformly with respect to $t$, and
$u:$\ $\mathbb{Q}_{p}^{n}\times\lbrack0,T]\rightarrow\mathbb{C}$ is an unknown function.\ 

We say that $u(x,t)$ is a \textit{solution of }(\ref{Cauchy1}), if $u(x,t)$ is
continuous in $(x,t)$, $u(\cdot,t)\in Dom(W_{\alpha})$ for $t\in\lbrack0,T],$
$u(x,\cdot)$ is continuously differentiable for $t\in(0,T]$, $u(x,t)\in
\mathcal{\mathfrak{M}}_{\lambda}\ $uniformly in $t$, and $\ u$ satisfies
(\ref{Cauchy1}) \ for all $t>0.$

Cauchy problem (\ref{Cauchy1}) was studied in \cite{Ch-Z} using semigroup
theory. In this article, we study this problem in the space
$\mathcal{\mathfrak{M}}_{\lambda}$, which is not contained in $L^{\rho}$ for
any $\rho\in\left[  1,\infty\right]  $, and thus we cannot use semigroup
theory, see e.g. \cite[Theorem 6.5]{Ch-Z}.

We define
\begin{equation}
Z(x,t;w_{\alpha},\kappa):=Z(x,t)=\underset{\mathbb{Q}_{p}^{n}}{\int}e^{-\kappa
tA_{w_{\alpha}}(\left\Vert \xi\right\Vert _{p})}\Psi(x\cdot\xi)d^{n}\xi,
\label{Z}%
\end{equation}
\ for $t>0$ \ and \ $x\in\mathbb{Q}_{p}^{n}$. Notice that, $Z(x,t)=\mathcal{F}%
_{\xi\rightarrow x}^{-1}[e^{-\kappa tA_{w_{\alpha}}(\left\Vert \xi\right\Vert
_{p})}]\in L^{1}\cap L^{2}$ for $t>0$, since \ $C^{\prime}\left\Vert
\xi\right\Vert _{p}^{\alpha-n}\leq$\ $A_{w_{\alpha}}(\left\Vert \xi\right\Vert
_{p})\leq C^{\prime\prime}\left\Vert \xi\right\Vert _{p}^{\alpha-n}$,\ cf.
\cite[Lemma 3.4]{Ch-Z}. Furthermore, $Z(x,t)\geq0$, for $t>0$, $x\in
\mathbb{Q}_{p}^{n}$, cf. \cite[Theorem 4.3 (i)]{Ch-Z}. These functions are
called \textit{heat kernels}. When considering $Z(x,t)$ as a function of $x$
for $t$ fixed we will write $Z_{t}(x)$.

We set%
\begin{align*}
u_{1}(x,t)  &  :=\underset{\mathbb{Q}_{p}^{n}}{\int}Z(x-y,t)\varphi
(y)d^{n}y\text{,}\\
u_{2}(x,t)  &  :=%
{\displaystyle\int\limits_{0}^{t}}
\underset{\mathbb{Q}_{p}^{n}}{\int}Z(x-y,t-\theta)f(y,\theta)d^{n}yd\theta,
\end{align*}
for $\varphi,f\in\mathcal{\mathfrak{M}}_{\lambda}$ with $\ \alpha-n>\lambda$,
for $0\leq t\leq T$, and $\ x\in\mathbb{Q}_{p}^{n}$.

The main result of this section is the following:

\begin{theorem}
\label{Thm1}The function%
\[
u(x,t)=u_{1}(x,t)+u_{2}(x,t)
\]
is a solution of Cauchy Problem (\ref{Cauchy1}).
\end{theorem}

The proof requires several steps.

\subsection{\textbf{Claim} $u(x,t)\in\mathcal{\mathfrak{M}}_{\lambda}$}

In order to prove this claim, we need some preliminary results.

\begin{remark}
\label{Remark}The function $Z_{t}(x)$ is radial since it is the
inverse\ Fourier transform of\ the radial function $e^{-\kappa tA_{w_{\alpha}%
}(\left\Vert \xi\right\Vert _{p})}$. Then $Z_{t}(x)$ is locally constant in
$\mathbb{Q}_{p}^{n}\backslash\{0\}.$ Furthermore, \ $Z_{t}(x+y)=Z_{t}(x)$ if
$\left\Vert y\right\Vert _{p}<\left\Vert x\right\Vert _{p}$ for any
$y\in\mathbb{Q}_{p}^{n}$ and $x\in\mathbb{Q}_{p}^{n}\backslash\{0\}$, and
$t>0$.
\end{remark}

\begin{lemma}
\label{lemmaZ} There exist positive constants $\ C_{1},C_{2}$ such that
$Z(x,t)$\ satisfies the following conditions:

(i) $Z(x,t)\leq C_{1}t^{-\frac{n}{\alpha-n}}$, for $t>0$ and $\ x\in
\mathbb{Q}_{p}^{n}$;

(ii) $Z(x,t)\leq C_{2}t\left\Vert x\right\Vert _{p}^{-\alpha}$, for $t>0$ and
$\ x\in\mathbb{Q}_{p}^{n}\backslash\{0\}$;

(iii) $Z(x,t)\leq\max\{2^{\alpha}C_{1},2^{\alpha}C_{2}\}t\left(  \left\Vert
x\right\Vert _{p}+t^{\frac{1}{\alpha-n}}\right)  ^{-\alpha}$, for $t>0$ and
$\ x\in\mathbb{Q}_{p}^{n}$;

(iv) $%
{\textstyle\int\nolimits_{\mathbb{Q}_{p}^{n}}}
Z(x,t)d^{n}x=1$, for $t>0$.

\begin{proof}
(i) By (\ref{Z}) and \ Lemma 3.4 in \cite{Ch-Z},
\[
Z(x,t)\leq\underset{\mathbb{Q}_{p}^{n}}{\int}e^{-\kappa tA_{w_{\alpha}%
}(\left\Vert \xi\right\Vert _{p})}d^{n}\xi\leq\underset{\mathbb{Q}_{p}^{n}%
}{\int}e^{-C_{0}t\left\Vert \xi\right\Vert _{p}^{\alpha-n}}d^{n}\xi.
\]

Let\ $m$ be an integer such that $p^{m-1}\leq t^{\frac{1}{\alpha-n}}\leq
p^{m}$, then%
\[
Z(x,t)\leq\underset{\mathbb{Q}_{p}^{n}}{\int}e^{-C_{0}t\left\Vert
p^{-(m-1)}\xi\right\Vert _{p}^{\alpha-n}}d^{n}\xi,
\]
now, by changing variables as\ $z=p^{-(m-1)}\xi$, we have%
\[
Z(x,t)\leq p^{-(m-1)n}\underset{\mathbb{Q}_{p}^{n}}{\int}e^{-C_{0}t\left\Vert
z\right\Vert _{p}^{\alpha-n}}d^{n}z\leq C_{1}t^{-\frac{n}{\alpha-n}}.
\]
(ii)\ It follows\ from\ \cite[Lemma 4.1]{Ch-Z}.\ (iii) The results is obtained
from the two following inequalities. If $\left\Vert x\right\Vert _{p}\geq
t^{\frac{1}{\alpha-n}}$, then $\left\Vert x\right\Vert _{p}\geq\frac
{\left\Vert x\right\Vert _{p}}{2}+\frac{t^{\frac{1}{\alpha-n}}}{2}$\ and
\ $\left\Vert x\right\Vert _{p}^{-\alpha}\leq2^{\alpha}\left(  \left\Vert
x\right\Vert _{p}+t^{\frac{1}{\alpha-n}}\right)  ^{-\alpha}$,\ multiplying
by$\ C_{2}t$\ and using (ii),
\[
Z(x,t)\leq2^{\alpha}C_{2}t\left(  \left\Vert x\right\Vert _{p}+t^{\frac
{1}{\alpha-n}}\right)  ^{-\alpha}.
\]
If $\left\Vert x\right\Vert _{p}\leq t^{\frac{1}{\alpha-n}}$, then
$\frac{\left\Vert x\right\Vert _{p}}{2}+\frac{t^{\frac{1}{\alpha-n}}}{2}\leq
t^{\frac{1}{\alpha-n}}$ and \ $\left(  \left\Vert x\right\Vert _{p}%
+t^{\frac{1}{\alpha-n}}\right)  ^{-\alpha}\geq2^{-\alpha}t^{\frac{-\alpha
}{\alpha-n}}=2^{-\alpha}t^{-1-\frac{n}{\alpha-n}}$, multiplying by $C_{1}$ and
using (i),\
\[
Z(x,t)\leq2^{\alpha}C_{1}t\left(  \left\Vert x\right\Vert _{p}+t^{\frac
{1}{\alpha-n}}\right)  ^{-\alpha}.
\]
(iv) By (iii), $Z_{t}(x)\in L^{1}(\mathbb{Q}_{p}^{n})$ \ for $t>0$.\ Now, the
announced identity follows by applying the Fourier inversion formula.
\end{proof}
\end{lemma}

\begin{proposition}
[\cite{R-Zu}, Proposition 2]\label{Conv} If $b>0$, $0\leq\lambda<\alpha$, and
$x\in\mathbb{Q}_{p}^{n}$, then%
\[
\underset{\mathbb{Q}_{p}^{n}}{\int}\left(  b+\left\Vert x-\xi\right\Vert
_{p}\right)  ^{-\alpha-n}\left\Vert \xi\right\Vert _{p}^{\lambda}d^{n}\xi\leq
Cb^{-\alpha}\left(  1+\left\Vert x\right\Vert _{p}^{\lambda}\right)  ,
\]
where the constant $C$ does not depend on $b$ or $x$.
\end{proposition}

\begin{lemma}
\label{u1u2}The functions $u_{1},u_{2}$ belong to $\mathcal{\mathfrak{M}%
}_{\lambda}$ uniformly in $t$, for $\lambda+n<\alpha$.
\end{lemma}

\begin{proof}
By Lemma \ref{lemmaZ} (iii), and Proposition \ref{Conv},%
\begin{align*}
\left\vert u_{1}(x,t)\right\vert  &  \leq\underset{\mathbb{Q}_{p}^{n}}{\int
}Z(x-y,t)\left\vert \varphi(y)\right\vert d^{n}y\leq C\underset{\mathbb{Q}%
_{p}^{n}}{\int}t\left(  t^{\frac{1}{\alpha-n}}+\left\Vert x-y\right\Vert
_{p}\right)  ^{-\alpha}\left(  1+\left\Vert y\right\Vert _{p}^{\lambda
}\right)  d^{n}y\\
&  \leq C^{\prime}\left(  1+\left\Vert x\right\Vert _{p}^{\lambda}\right)  .
\end{align*}
On the other hand, since%
\[
u_{1}(x,t)=\underset{\mathbb{Q}_{p}^{n}}{\int}Z(w,t)\varphi(x-w)d^{n}w,
\]
$u_{1}$ is locally constant and $l\left(  u_{1}\right)  =l\left(
\varphi\right)  $\ uniformly in $t$. The proof for $u_{2}$ is similar.
\end{proof}

\begin{remark}
\label{nota2}Notice that $u_{1}$, $u_{2}$, $\boldsymbol{W}_{\gamma}u_{1}$,
$\boldsymbol{W}_{\gamma}u_{2}$ $\in\mathcal{\mathfrak{M}}_{\lambda}$, for any
$\gamma$ satisfying $\lambda+n<\gamma\leq\alpha$.
\end{remark}

\subsection{\textbf{Claim} $u(x,t)$ satisfies the initial condition}

This claim follows from Le\-mma \ref{u1u2} by using the following result.

\begin{lemma}
\label{CI}If $\varphi\in\mathcal{\mathfrak{M}}_{\lambda}$, with $\alpha
>\lambda+n$, then%
\[
\underset{t\rightarrow0^{+}}{\lim}\underset{\mathbb{Q}_{p}^{n}}{\int}%
Z(x-\xi,t)\varphi(\xi)d^{n}\xi=\varphi(x).
\]

\end{lemma}

\begin{proof}
By Lemma \ref{lemmaZ} (iv),%
\begin{equation}
\underset{\mathbb{Q}_{p}^{n}}{\int}Z(x-\xi,t)\varphi(\xi)d^{n}\xi
=\underset{\mathbb{Q}_{p}^{n}}{\int}Z(x-\xi,t)\left[  \varphi(\xi
)-\varphi(x)\right]  d^{n}\xi+\varphi(x). \label{phi-phi}%
\end{equation}

Now, by Lemma \ref{lemmaZ} (iii) and the local constancy of $\varphi$,
\begin{multline*}
\underset{\mathbb{Q}_{p}^{n}}{\int}Z(x-\xi,t)\left[  \varphi(\xi
)-\varphi(x)\right]  d^{n}\xi\leq\underset{\left\Vert x-\xi\right\Vert
_{p}\geq p^{l}}{Ct\int}(t^{\frac{1}{\alpha-n}}+\left\Vert x-\xi\right\Vert
_{p})^{-\alpha}\left\vert \varphi(\xi)-\varphi(x)\right\vert d^{n}\xi\\
\leq\underset{\left\Vert z\right\Vert _{p}\geq p^{l}}{Ct\int}(t^{\frac
{1}{\alpha-n}}+\left\Vert z\right\Vert _{p})^{-\alpha}\left\vert
\varphi(x-z)-\varphi(x)\right\vert d^{n}z\\
\leq\underset{\left\Vert z\right\Vert _{p}\geq p^{l}}{Ct\int}\left\Vert
z\right\Vert _{p}^{-\alpha}(1+\left\Vert x-z\right\Vert ^{\lambda}%
)d^{n}z+C^{\prime}t\left\vert \varphi(x)\right\vert \leq th\left(  x\right)  .
\end{multline*}
The formula is obtained by taking limit when $t\rightarrow0^{+}$ in
(\ref{phi-phi}).
\end{proof}

\subsection{Claim $u(x,t)$ is a solution of Cauchy problem (\ref{Cauchy1})}

The proof of this claim is a consequence of Corollary \ref{coro1}, Lemmas
\ref{u1} and \ref{u2}.\ Several preliminary results are required.

\begin{lemma}
\label{lemmaZ'} There exist positive constants$\ C_{3},C_{4}$ such that
$Z(x,t)$ \ satisfies the following conditions:

(i) $\frac{\partial Z(x,t)}{\partial t}=-\kappa%
{\textstyle\int\nolimits_{\mathbb{Q}_{p}^{n}}}
A_{w_{\alpha}}(\left\Vert \xi\right\Vert _{p})e^{-\kappa tA_{w_{\alpha}%
}(\left\Vert \xi\right\Vert _{p})}\Psi(x\cdot\xi)d^{n}\xi$, for $t>0$ and
$\ x\in\mathbb{Q}_{p}^{n}$;

(ii) $\left\vert \frac{\partial Z(x,t)}{\partial t}\right\vert \leq
C_{3}t^{-\frac{\alpha}{\alpha-n}}$, for $t>0$ and $\ x\in\mathbb{Q}_{p}^{n}$;

(iii) $\left\vert \frac{\partial Z(x,t)}{\partial t}\right\vert \leq
C_{4}t\left\Vert x\right\Vert _{p}^{n-2\alpha}$, for $t>0$ and $\ x\in
\mathbb{Q}_{p}^{n}\backslash\{0\}$;

(iv) $\left\vert \frac{\partial Z(x,t)}{\partial t}\right\vert \leq2^{\alpha
}C_{3}\left(  \left\Vert x\right\Vert _{p}+t^{\frac{1}{\alpha-n}}\right)
^{-\alpha}$, for $t>0$ and $x\in\mathbb{Q}_{p}^{n}\setminus\{0\}$.

\begin{proof}
(i) The\ formula is obtained by the Lebesgue \ Dominated Convergence Theorem,
and the fact that\ $-\kappa A_{w_{\alpha}}(\left\Vert \xi\right\Vert
_{p})e^{-\kappa\tau A_{w_{\alpha}}(\left\Vert \xi\right\Vert _{p})}\Psi
(x\cdot\xi)\in L^{1}(\mathbb{Q}_{p}^{n})$, for $\tau>0$ fixed,\ cf.
\cite[Lemma 3.4]{Ch-Z}.\ (ii) By using (i) and Lemma 3.4 in \cite{Ch-Z},%
\[
\left\vert \frac{\partial Z(x,t)}{\partial t}\right\vert \leq\underset
{\mathbb{Q}_{p}^{n}}{\int}C_{1}\left\Vert \xi\right\Vert _{p}^{\alpha
-n}e^{-\kappa C_{2}t\left\Vert \xi\right\Vert _{p}^{\alpha-n}}d^{n}\xi.
\]
We now pick an integer $m$ such that $\ p^{m-1}\leq t^{\frac{1}{\alpha-n}}\leq
p^{m}$, and proceed as in the proof of Lemma \ref{lemmaZ} (i), to obtain%
\[
\left\vert \frac{\partial Z(x,t)}{\partial t}\right\vert \leq C_{1}%
p^{-(m-1)n-(m-1)(\alpha-n)}\underset{\mathbb{Q}_{p}^{n}}{\int}\left\Vert
z\right\Vert _{p}^{\alpha-n}e^{-\kappa C_{2}t\left\Vert z\right\Vert
_{p}^{\alpha-n}}d^{n}z\leq C_{3}t^{-\frac{\alpha}{\alpha-n}}.
\]
(iii) Set $\left\Vert x\right\Vert _{p}=p^{\beta}$. Now, since\ $A_{w_{\alpha
}}(\left\Vert \xi\right\Vert _{p})e^{-\kappa tA_{w_{\alpha}}(\left\Vert
\xi\right\Vert _{p})}\in L^{1}\cap L^{2}$ for $t>0$, then \ $\frac{\partial
Z(x,t)}{\partial t}\in L^{1}\cap L^{2}$\ for $t>0$, and by applying the
formula \ for the Fourier Transform of a radial function, we get
\begin{multline*}
\frac{\partial Z(x,t)}{\partial t}=\left\Vert x\right\Vert _{p}^{-n}\\
\times\left(  (1-p^{-n})\sum_{j=0}^{\infty}A_{w_{\alpha}}(p^{-\beta
-j})e^{-\kappa tA_{w_{\alpha}}(p^{-\beta-j})}p^{-nj}-A_{w_{\alpha}}%
(p^{-\beta+1})e^{-\kappa tA_{w_{\alpha}}(p^{-\beta+1})}\right)  .
\end{multline*}
Now, by using that\ $A_{w_{\alpha}}(\xi)$ is a decreasing function of
$ord(\xi)$,
\begin{align*}
\left\vert \frac{\partial Z(x,t)}{\partial t}\right\vert  &  \leq\left\Vert
x\right\Vert _{p}^{-n}A_{w_{\alpha}}(p^{-\beta+1})\left\vert (1-p^{-n}%
)\sum_{j=0}^{\infty}p^{-nj}-e^{-\kappa tA_{w_{\alpha}}(p^{-\beta+1}%
)}\right\vert \\
&  \leq\left\Vert x\right\Vert _{p}^{-n}A_{w_{\alpha}}(p^{-\beta+1})\left(
1-e^{-\kappa tA_{w_{\alpha}}(p^{-\beta+1})}\right)
\end{align*}
By using Mean Value Theorem and Lemma 3.4 in \cite{Ch-Z}, we have%
\[
\left\vert \frac{\partial Z(x,t)}{\partial t}\right\vert \leq C_{4}\left\Vert
x\right\Vert _{p}^{n-2\alpha}t.
\]
(iv) If $\left\Vert x\right\Vert _{p}\leq t^{\frac{1}{\alpha-n}}$,
then\ $\frac{\left\Vert x\right\Vert _{p}}{2}+\frac{t^{\frac{1}{\alpha-n}}}%
{2}\leq t^{\frac{1}{\alpha-n}}$\ and\ $t^{\frac{-\alpha}{\alpha-n}}%
\leq2^{\alpha}\left(  \left\Vert x\right\Vert _{p}+t^{\frac{1}{\alpha-n}%
}\right)  ^{-\alpha}$, multiplying by $\ C_{3}$ and using (ii), we have
\[
\left\vert \frac{\partial Z(x,t)}{\partial t}\right\vert \leq2^{\alpha}%
C_{3}\left(  \left\Vert x\right\Vert _{p}+t^{\frac{1}{\alpha-n}}\right)
^{-\alpha}.
\]
Now, if $\left\Vert x\right\Vert _{p}\geq t^{\frac{1}{\alpha-n}}$, by using
(iii),
\begin{equation}
\left\vert \frac{\partial Z(x,t)}{\partial t}\right\vert \leq C_{3}\left\Vert
x\right\Vert _{p}^{-\alpha}, \label{Z'<X}%
\end{equation}
and since$\left\Vert x\right\Vert _{p}\geq t^{\frac{1}{\alpha-n}}$,
then$\ \left\Vert x\right\Vert _{p}\geq\left(  \frac{\left\Vert x\right\Vert
_{p}}{2}+\frac{t^{\frac{1}{\alpha-n}}}{2}\right)  \ $and $2^{\alpha}\left(
\left\Vert x\right\Vert _{p}+t^{\frac{1}{\alpha-n}}\right)  ^{-\alpha}%
\geq\left\Vert x\right\Vert _{p}^{-\alpha}$, multiplying by $C_{3}$ and using
(\ref{Z'<X}), we have\
\[
\left\vert \frac{\partial Z(x,t)}{\partial t}\right\vert \leq2^{\alpha}%
C_{3}\left(  \left\Vert x\right\Vert _{p}+t^{\frac{1}{\alpha-n}}\right)
^{-\alpha}.
\]

\end{proof}
\end{lemma}

\begin{lemma}
\label{lemmaWZ}$\left(  \mathbf{W}_{\gamma}Z_{t}\right)  (x)$, with
$\gamma\leq\alpha$, satisfies the following conditions:

(i) $\left(  \mathbf{W}_{\gamma}Z_{t}\right)  (x)=-%
{\textstyle\int\nolimits_{\mathbb{Q}_{p}^{n}}}
A_{w_{\gamma}}(\left\Vert \xi\right\Vert _{p})e^{-\kappa tA_{w_{\alpha}%
}(\left\Vert \xi\right\Vert _{p})}\Psi(x\cdot\xi)d^{n}\xi$, for $t>0$ and
$x\in\mathbb{Q}_{p}^{n}$;

(ii) $\left\vert \left(  \boldsymbol{W}_{\gamma}Z_{t}\right)  (x)\right\vert
\leq2^{\gamma}C\left(  \left\Vert x\right\Vert _{p}+t^{\frac{1}{\alpha-n}%
}\right)  ^{-\gamma}$, for $t>0$ and $x\in\mathbb{Q}_{p}^{n}$ and some
positive constant $C$;

(iii) $%
{\textstyle\int\nolimits_{\mathbb{Q}_{p}^{n}}}
\left(  \boldsymbol{W}_{\gamma}Z_{t}\right)  (x)d^{n}x=0.$
\end{lemma}

\begin{proof}
(i) Define
\begin{equation}
Z_{t}^{(M)}(x)=\underset{\left\Vert \eta\right\Vert _{p}\leq p^{M}}{\int}%
\Psi(x\cdot\eta)e^{-\kappa tA_{w_{\alpha}}(\left\Vert \eta\right\Vert _{p}%
)}d^{n}\eta\text{,\ for\ }M\in\mathbb{N}\text{.} \label{Z(M)}%
\end{equation}
This function is locally constant on $\mathbb{Q}_{p}^{n}$. Indeed, if
$\left\Vert \xi\right\Vert _{p}\leq p^{-M},$ then $Z_{t}^{(M)}(x+\xi
)=Z_{t}^{(M)}(x)$. Furthermore, $Z_{t}^{(M)}(x)$ is bounded, and thus
$Z_{t}^{(M)}(x)\in\mathcal{\mathfrak{M}}_{0}\subset$ $Dom(\boldsymbol{W}%
_{\gamma})$. We now use formula (\ref{W}) and Fubini's Theorem to
compute$\ (\mathbf{W}_{\gamma}Z_{t}^{(M)})(x)$ as follows:
\begin{align*}
(\mathbf{W}_{\gamma}Z_{t}^{(M)})(x)  &  =\underset{\mathbb{Q}_{p}^{n}}{\int
}\frac{Z_{t}^{(M)}(x-\xi)-Z_{t}^{(M)}(x)}{w_{\gamma}(\left\Vert \xi\right\Vert
_{p})}d^{n}\xi\\
&  =\underset{\left\Vert \xi\right\Vert _{p}>p^{-M}}{\int}\text{ }%
\underset{\left\Vert \eta\right\Vert _{p}\leq p^{M}}{\int}e^{-\kappa
tA_{w_{\alpha}}(\left\Vert \eta\right\Vert _{p})}\Psi(x\cdot\eta)\frac{\left(
\Psi(\xi\cdot\eta)-1\right)  }{w_{\gamma}(\left\Vert \xi\right\Vert _{p}%
)}d^{n}\eta d^{n}\xi\\
&  =\underset{\left\Vert \eta\right\Vert _{p}\leq p^{M}}{\int}e^{-\kappa
tA_{w_{\alpha}}(\left\Vert \eta\right\Vert _{p})}\Psi(x\cdot\eta
)\underset{\left\Vert \xi\right\Vert _{p}>p^{-M}}{\int}\frac{\left(  \Psi
(\xi\cdot\eta)-1\right)  }{w_{\gamma}(\left\Vert \xi\right\Vert _{p})}d^{n}\xi
d^{n}\eta\\
&  =-\underset{\left\Vert \eta\right\Vert _{p}\leq p^{M}}{\int}e^{-\kappa
tA_{w_{\alpha}}(\left\Vert \eta\right\Vert _{p})}\Psi(x\cdot\eta)A_{w_{\gamma
}}(\left\Vert \eta\right\Vert _{p})d^{n}\eta.
\end{align*}
By using that \ $e^{-\kappa tA_{w_{\alpha}}(\left\Vert \xi\right\Vert _{p}%
)}A_{w_{\gamma}}(\left\Vert \xi\right\Vert _{p})\in L^{1}(\mathbb{Q}_{p}^{n})$
for $t>0$, cf. \ \cite[Lemma 3.4]{Ch-Z} and the Dominated Convergence Theorem,
we obtain
\begin{equation}
\lim_{M\rightarrow\infty}(\mathbf{W}_{\gamma}Z_{t}^{(M)})(x)=-\underset
{\mathbb{Q}_{p}^{n}}{\int}A_{w_{\gamma}}(\left\Vert \eta\right\Vert
_{p})e^{-\kappa tA_{w_{\alpha}}(\left\Vert \eta\right\Vert _{p})}\Psi
(x\cdot\eta)d^{n}\eta. \label{WZ(M)}%
\end{equation}
On the other hand, by fixing $x\neq0$ and for $t>0$, $Z_{t}(x-\xi)-Z_{t}(x)$
is locally constant, cf. Remark \ref{Remark}, and bounded, cf. Lemma
\ref{lemmaZ} (iii), then $(\mathbf{W}_{\gamma}Z_{t})(x)$ is well-defined, and
since $Z_{t}^{(M)}\left(  x\right)  $ is radial,
\[
(\mathbf{W}_{\gamma}Z_{t}^{(M)})(x)=\underset{\left\Vert \xi\right\Vert
_{p}>\left\Vert x\right\Vert _{p}}{\int}\frac{Z_{t}^{(M)}(x-\xi)-Z_{t}%
^{(M)}(x)}{w_{\gamma}(\left\Vert \xi\right\Vert _{p})}d^{n}\xi,
\]
and by Dominated Convergence Theorem, $\lim_{M\rightarrow\infty}%
(\mathbf{W}_{\gamma}Z_{t}^{(M)})(x)=(\mathbf{W}_{\gamma}Z_{t})(x)$. Therefore
by (\ref{WZ(M)}), we have%
\[
(\mathbf{W}_{\gamma}Z_{t})(x)=-\underset{\mathbb{Q}_{p}^{n}}{\int}%
A_{w_{\gamma}}(\left\Vert \eta\right\Vert _{p})e^{-\kappa tA_{w_{\alpha}%
}(\left\Vert \eta\right\Vert _{p})}\Psi(x\cdot\eta)d^{n}\eta.
\]
Finally,\ we note the right-hand side in the above formula is continuous at
\ $x=0$.

(ii) By (i) and Lemma 3.4 \ in \ \cite{Ch-Z},%
\[
\left\vert (\mathbf{W}_{\gamma}Z_{t})(x)\right\vert \leq C_{0}\underset
{\mathbb{Q}_{p}^{n}}{\int}\left\Vert \xi\right\Vert _{p}^{\gamma-n}e^{-\kappa
C_{1}t\left\Vert \xi\right\Vert _{p}^{\alpha-n}}d^{n}\xi.
\]
We now pick\ an integer $m$ such that $\ p^{m-1}\leq t^{\frac{1}{\alpha-n}%
}\leq p^{m}$, and proceed as in the proof of Lemma \ref{lemmaZ} (i), to
obtain
\begin{equation}
\left\vert (\mathbf{W}_{\gamma}Z_{t})(x)\right\vert \leq Ct^{-\frac{\gamma
}{\alpha-n}}. \label{Wz<ct}%
\end{equation}
Now, \ if $\ \left\Vert x\right\Vert _{p}\leq t^{\frac{1}{\alpha-n}},$ then
\ $\frac{\left\Vert x\right\Vert _{p}}{2}+\frac{t^{\frac{1}{\alpha-n}}}{2}\leq
t^{\frac{1}{\alpha-n}}$ and \ $t^{\frac{-\gamma}{\alpha-n}}\leq2^{\gamma
}\left(  \left\Vert x\right\Vert _{p}+t^{\frac{1}{\alpha-n}}\right)
^{-\gamma}$, multiplying by $\ C$\ and \ by using (\ref{Wz<ct}), we have
\[
\left\vert (\mathbf{W}_{\gamma}Z_{t})(x)\right\vert \leq2^{\gamma}C\left(
\left\Vert x\right\Vert _{p}+t^{\frac{1}{\alpha-n}}\right)  ^{-\gamma}.
\]
On the other hand, let $\left\Vert x\right\Vert _{p}=p^{\beta}$, since
\ $A_{w_{\gamma}}(\left\Vert \xi\right\Vert _{p})e^{-\kappa tA_{w_{\alpha}%
}(\left\Vert \xi\right\Vert _{p})}\in L^{1}\cap L^{2}$ for $t>0$, then
\ $(\mathbf{W}_{\gamma}Z_{t})(x)\in L^{1}\cap L^{2}$\ for $t>0$, by proceeding
as in the proof of Lemma \ref{lemmaZ'} (iii), we obtain%
\[
\left\vert (\mathbf{W}_{\gamma}Z_{t})(x)\right\vert \leq Ct\left\Vert
x\right\Vert _{p}^{n-\alpha-\gamma}.
\]
Now, if $\left\Vert x\right\Vert _{p}\geq t^{\frac{1}{\alpha-n}}$, then
\begin{equation}
\left\vert (\mathbf{W}_{\gamma}Z_{t})(x)\right\vert \leq C\left\Vert
x\right\Vert _{p}^{-\gamma}. \label{WZ<CX}%
\end{equation}
If $\left\Vert x\right\Vert _{p}\geq t^{\frac{1}{\alpha-n}},$ then
$\ \left\Vert x\right\Vert _{p}\geq\left(  \frac{\left\Vert x\right\Vert _{p}%
}{2}+\frac{t^{\frac{1}{\alpha-n}}}{2}\right)  $ and $2^{\gamma}\left(
\left\Vert x\right\Vert _{p}+t^{\frac{1}{\alpha-n}}\right)  ^{-\gamma}%
\geq\left\Vert x\right\Vert _{p}^{-\gamma}$, multiplying by $C$ and using
(\ref{WZ<CX}), we have\
\[
\left\vert (\mathbf{W}_{\gamma}Z_{t})(x)\right\vert \leq2^{\gamma}C\left(
\left\Vert x\right\Vert _{p}+t^{\frac{1}{\alpha-n}}\right)  ^{-\gamma}.
\]
(iii) It follows from (i) by the inversion formula for the Fourier transform.
\end{proof}

\begin{corollary}
\label{coro1}$\frac{\partial Z(x,t)}{\partial t}=\kappa\cdot\left(
\mathbf{W}_{\alpha}Z_{t}\right)  (x)$ \ for $t>0$ and $x\in\mathbb{Q}_{p}^{n}$.
\end{corollary}

\begin{proof}
The formula follows from Lemma \ref{lemmaZ'} (i) and Lemma \ref{lemmaWZ} (i).
\end{proof}

\begin{proposition}
\label{u1}Assume that $\varphi\in\mathcal{\mathfrak{M}}_{\lambda}$, then the
following assertions hold:

(i) $\frac{\partial u_{1}}{\partial t}(x,t)=%
{\textstyle\int\nolimits_{\mathbb{Q}_{p}^{n}}}
\frac{\partial Z(x-y,t)}{\partial t}\varphi(y)d^{n}y$,\ for $t>0$
and$\ x\in\mathbb{Q}_{p}^{n}\backslash\{0\}$;

(ii) $(\mathbf{W}_{\gamma}u_{1})(x,t)=%
{\textstyle\int\nolimits_{\mathbb{Q}_{p}^{n}}}
(\mathbf{W}_{\gamma}Z_{t})(x-y)\varphi(y)d^{n}y$,\ for $n+\lambda<\gamma
\leq\alpha$, $t>0$ and $x\in\mathbb{Q}_{p}^{n}\backslash\{0\}$.
\end{proposition}

\begin{proof}
(i) By using the Mean Value Theorem, $\frac{\partial u_{1}}{\partial t}(x,t)$
equals%
\[
\lim_{h\rightarrow0}\underset{\mathbb{Q}_{p}^{n}}{\int}\left[  \frac
{Z(x-y,t+h)-Z(x-y,t)}{h}\right]  \varphi(y)d^{n}y=\lim_{h\rightarrow
0}\underset{\mathbb{Q}_{p}^{n}}{\int}\frac{\partial Z(x-y,\tau)}{\partial
t}\varphi(y)d^{n}y,
\]
where $\tau$ is between $t$ and $t+h$. Now, the result follows by applying the
Dominated Converge Theorem and Lemma \ref{lemmaZ'} (iv).

(ii) By Remark \ref{nota2}, if \ $n+\lambda<\gamma$, then $u_{1}\in
Dom(\mathbf{W}_{\gamma})$ for\ $t>0$. Then \ for any $L\in\mathbb{N}$, the
following integral exists:%
\begin{multline*}
\underset{\left\Vert y\right\Vert _{p}>p^{-L}}{\int}\frac{u_{1}(x-y,t)-u_{1}%
(x,t)}{w_{\gamma}(\left\Vert y\right\Vert _{p})}d^{n}y\\
=\underset{\left\Vert y\right\Vert _{p}>p^{-L}}{\int}\frac{1}{w_{\gamma
}(\left\Vert y\right\Vert _{p})}\underset{\mathbb{Q}_{p}^{n}}{\int}\left(
Z_{t}(x-y-\xi)-Z_{t}(x-\xi)\right)  \varphi(\xi)d^{n}\xi d^{n}y,
\end{multline*}
now, by using Fubini's Theorem, cf. Lemma \ref{lemmaZ} (iii),
\[
\underset{\mathbb{Q}_{p}^{n}}{\int}\varphi(\xi)\underset{\left\Vert
y\right\Vert _{p}>p^{-L}}{\int}\frac{\left(  Z_{t}(x-\xi-y)-Z_{t}%
(x-\xi)\right)  }{w_{\gamma}(\left\Vert y\right\Vert _{p})}d^{n}yd^{n}\xi.
\]

We now fix a positive integer $M$,\ such that $\left\Vert y\right\Vert
_{p}<p^{-L}<p^{-M}<\left\Vert x-\xi\right\Vert _{p}$, and use Remark
\ref{Remark},%
\begin{multline*}
(\boldsymbol{W}_{\gamma}u_{1})(x,t)=\lim_{L\rightarrow\infty}\underset
{\left\Vert y\right\Vert _{p}>p^{-L}}{\int}\frac{u_{1}(x-y,t)-u_{1}%
(x,t)}{w_{\gamma}(\left\Vert y\right\Vert _{p})}d^{n}y\\
=\underset{\left\Vert x-\xi\right\Vert _{p}>p^{-M}}{\int}\varphi
(\xi)(\boldsymbol{W}_{\gamma}Z_{t})(x-\xi)d^{n}\xi\\
+\lim_{L\rightarrow\infty}\underset{\left\Vert x-\xi\right\Vert _{p}\leq
p^{-M}}{\int}\varphi(\xi)\underset{\left\Vert y\right\Vert _{p}>p^{-L}}{\int
}\frac{\left(  Z_{t}(x-\xi-y)-Z_{t}(x-\xi)\right)  }{w_{\gamma}(\left\Vert
y\right\Vert _{p})}d^{n}yd^{n}\xi\\
=\underset{\left\Vert x-\xi\right\Vert _{p}>p^{-M}}{\int}\varphi
(\xi)(\boldsymbol{W}_{\gamma}Z_{t})(x-\xi)d^{n}\xi+\underset{\left\Vert
x-\xi\right\Vert _{p}\leq p^{-M}}{\int}\varphi(\xi)(\boldsymbol{W}_{\gamma
}Z_{t})(x-\xi)d^{n}\xi.
\end{multline*}

\end{proof}

Set
\[
u_{2}(x,t,\tau):=%
{\displaystyle\int\limits_{\tau}^{t}}
\underset{\mathbb{Q}_{p}^{n}}{\int}Z(x-y,t-\theta)f(y,\theta)d^{n}yd\theta,
\]
for $f\in\mathcal{\mathfrak{M}}_{\lambda}$ with $\ \alpha-n>\lambda$, for
$0\leq\tau\leq t\leq T$, and $\ x\in\mathbb{Q}_{p}^{n}$. By reasoning as in
the proof of Lemma \ref{u1u2}, we have $u_{2}(x,t,\tau)\in
\mathcal{\mathfrak{M}}_{\lambda}$ uniformly in $t$ and $\tau$.

\begin{proposition}
\label{u2}Assume that $f\in\mathcal{\mathfrak{M}}_{\lambda}$, with
$\alpha-n>\lambda$, then the following assertions hold:

(i) $\frac{\partial u_{2}}{\partial t}(x,t,\tau)=f(x,t)+%
{\textstyle\int\nolimits_{\tau}^{t}}
\left(
{\textstyle\int\nolimits_{\mathbb{Q}_{p}^{n}}}
\frac{\partial Z(x-y,t-\theta)}{\partial t}\left[  f(y,\theta)-f(x,\theta
)\right]  d^{n}y\right)  d\theta$, for $t>0$ and$\ x\in\mathbb{Q}_{p}^{n}$;

(ii) $(\mathbf{W}_{\gamma}u_{2})(x,t,\tau)=%
{\textstyle\int\nolimits_{\tau}^{t}}
{\textstyle\int\nolimits_{\mathbb{Q}_{p}^{n}}}
(\mathbf{W}_{\gamma}Z)(x-y,t-\theta)f(y,\theta)d^{n}yd\theta$, for
$n+\lambda<\gamma\leq\alpha$, $t>0$ and$\ x\in\mathbb{Q}_{p}^{n}$.
\end{proposition}

\begin{proof}
Set%
\[
u_{2,h}(x,t,\tau):=%
{\displaystyle\int\limits_{\tau}^{t-h}}
\underset{\mathbb{Q}_{p}^{n}}{\int}Z(x-y,t-\theta)f(y,\theta)d^{n}%
yd\theta\text{, }0<h<t-\tau.
\]
By using a standard reasoning, one shows that
\begin{multline*}
\frac{\partial u_{2,h}}{\partial t}(x,t,\tau)=\\%
{\displaystyle\int\limits_{\tau}^{t-h}}
\underset{\mathbb{Q}_{p}^{n}}{\int}\frac{\partial Z(x-y,t-\theta)}{\partial
t}f(y,\theta)d^{n}yd\theta+\underset{\mathbb{Q}_{p}^{n}}{\int}%
Z(x-y,h)f(y,t-h)d^{n}yd\theta.
\end{multline*}

This formula can be rewritten as%
\begin{multline*}
\frac{\partial u_{2,h}}{\partial t}(x,t,\tau)=%
{\displaystyle\int\limits_{\tau}^{t-h}}
\underset{\mathbb{Q}_{p}^{n}}{\int}\frac{\partial Z(x-y,t-\theta)}{\partial
t}\left[  f(y,\theta)-f(x,\theta)\right]  d^{n}yd\theta\\
+%
{\displaystyle\int\limits_{\tau}^{t-h}}
f(x,\theta)\underset{\mathbb{Q}_{p}^{n}}{\int}\frac{\partial Z(x-y,t-\theta
)}{\partial t}d^{n}yd\theta+\underset{\mathbb{Q}_{p}^{n}}{\int}Z(x-y,h)\left[
f(y,t-h)-f(y,t)\right]  d^{n}y\\
+\underset{\mathbb{Q}_{p}^{n}}{\int}Z(x-y,h)f(y,t)d^{n}y.
\end{multline*}

The first integral contains no singularity at $t=\theta$ \ due to Lemma
\ref{lemmaZ'} (iv)\ and the local constancy of $\ f$. \ By Lemma \ref{lemmaZ}
(iv), \ the second integral is equal to zero. The third integral can be
written \ as a sum of the integrals over $\left\{  y\in\mathbb{Q}_{p}%
^{n}\text{ }|\text{ }\left\Vert x-y\right\Vert _{p}\geq p^{M}\right\}  $ and
the complement of this set, one of these integrals is estimated on the basis
of the uniform continuity of $f$, while the other contains no singularity, see
Lemma \ref{lemmaZ'} (iv). Finally, the fourth integral tends to $\ f(x,t)$ as
\ $h\rightarrow0^{+}$, cf.\ Lemma \ref{CI}.

(ii) By Lemma \ref{u1u2}, $\mathbf{W}_{\gamma}u_{2,h}$ is well-defined for any
$\gamma$ satisfying $n+\lambda<\gamma\leq\alpha$. Then, for any $\ L\in
\mathbb{N}$, the following integral exists:%
\begin{align}
&
{\displaystyle\int\limits_{\left\Vert \xi\right\Vert _{p}\geq p^{-L}}}
\frac{\left[  u_{2,h}(x-\xi,t,\tau)-u_{2,h}(x,t,\tau)\right]  }{w_{\gamma
}(\left\Vert \xi\right\Vert _{p})}d^{n}\xi\label{breaking}\\
&  =%
{\displaystyle\int\limits_{\tau}^{t-h}}
\underset{\mathbb{Q}_{p}^{n}}{\int}%
{\displaystyle\int\limits_{\left\Vert \xi\right\Vert _{p}\geq p^{-L}}}
\frac{\left[  Z(x-\xi-y,t-\theta)-Z(x-y,t-\theta)\right]  }{w_{\gamma
}(\left\Vert \xi\right\Vert _{p})}f(y,\theta)d^{n}\xi d^{n}yd\theta.\nonumber
\end{align}

On the other hand, by Fubini's Theorem,%
\begin{multline*}%
{\displaystyle\int\limits_{\left\Vert \xi\right\Vert _{p}\geq p^{-L}}}
\frac{\left[  Z(x-\xi-y,t-\theta)-Z(x-y,t-\theta)\right]  }{w_{\gamma
}(\left\Vert \xi\right\Vert _{p})}d^{n}\xi\\
=\underset{\mathbb{Q}_{p}^{n}}{\int}\Psi((x-y)\cdot\eta)e^{-\kappa
(t-\theta)A_{w_{\gamma}}(\left\Vert \eta\right\Vert _{p})}P_{k}(\eta)d^{n}%
\eta,
\end{multline*}
where%
\[
P_{k}(\eta)=%
{\displaystyle\int\limits_{\left\Vert \xi\right\Vert _{p}\geq p^{-L}}}
\frac{\left[  \Psi(-\xi\cdot\eta)-1\right]  }{w_{\gamma}(\left\Vert
\xi\right\Vert _{p})}d^{n}\xi.
\]

A simple calculation shows that%
\[
\left\vert P_{k}(\eta)\right\vert \leq C^{\prime}\left\Vert \eta\right\Vert
_{p}^{\gamma-n},
\]
and then%
\[%
{\displaystyle\int\limits_{\left\Vert \xi\right\Vert _{p}\geq p^{-L}}}
\frac{\left[  Z(x-\xi-y,t-\theta)-Z(x-y,t-\theta)\right]  }{w_{\gamma
}(\left\Vert \xi\right\Vert _{p})}d^{n}\xi\leq C,
\]
where the constant does not depend on $x$, $t\geq h+\tau$, $L$.

Now, by expressing the right integral of (\ref{breaking}) as%
\begin{multline}%
{\displaystyle\int\limits_{\tau}^{t-h}}
\underset{\left\Vert x-\xi\right\Vert _{p}>p^{-M}}{\int}\text{ }%
{\displaystyle\int\limits_{\left\Vert \xi\right\Vert _{p}\geq p^{-L}}}
\frac{\left[  Z(x-\xi-y,t-\theta)-Z(x-y,t-\theta)\right]  }{w_{\gamma
}(\left\Vert \xi\right\Vert _{p})}f(y,\theta)d^{n}\xi d^{n}yd\theta
\label{Wu2h2}\\
+%
{\displaystyle\int\limits_{\tau}^{t-h}}
\underset{\left\Vert x-\xi\right\Vert _{p}\leq p^{-M}}{\int}\text{ }%
{\displaystyle\int\limits_{\left\Vert \xi\right\Vert _{p}\geq p^{-L}}}
\frac{\left[  Z(x-\xi-y,t-\theta)-Z(x-y,t-\theta)\right]  }{w_{\gamma
}(\left\Vert \xi\right\Vert _{p})}f(y,\theta)d^{n}\xi d^{n}yd\theta\nonumber
\end{multline}

where $M$ is a positive integer, \ such that $\left\Vert \xi\right\Vert
_{p}<p^{-L}<p^{-M}<\left\Vert x-\xi\right\Vert _{p}$, and using the same
reasoning as in the final part of the proof of Proposition \ref{u1} (ii), we
obtain
\begin{equation}
(\boldsymbol{W}_{\gamma}u_{2,h})(x,t)=%
{\displaystyle\int\limits_{\tau}^{t-h}}
\underset{\mathbb{Q}_{p}^{n}}{\int}(\boldsymbol{W}_{\gamma}Z)(x-\xi
,t-\theta)f(y,\theta)d^{n}\xi d\theta. \label{Wgamau2h}%
\end{equation}

Now, by Lemma \ref{lemmaWZ} (ii), the fact that $\ f\in\mathcal{\mathfrak{M}%
}_{\lambda}$, Proposition \ref{Conv}, and the Dominated Convergence Theorem,
we can take limit \ as\ $h\rightarrow0^{+}$, which completes the proof when
$\gamma<\alpha.$ If $\gamma=\alpha$, formula (\ref{Wgamau2h}) remains valid.
By using Lemma \ref{lemmaWZ} (iii), formula (\ref{Wgamau2h}) can be rewritten
as%
\[
(\boldsymbol{W}_{\gamma}u_{2,h})(x,t)=%
{\displaystyle\int\limits_{\tau}^{t-h}}
\underset{\mathbb{Q}_{p}^{n}}{\int}(\boldsymbol{W}_{\gamma}Z)(x-\xi
,t-\theta)\left[  f(y,\theta)-f(x,\theta)\right]  d^{n}\xi d\theta.
\]

Now, by using the local constancy of\ $f$, we can justify the \ passage to the
\ limit as \ $h\rightarrow0^{+}$, which completes the proof.
\end{proof}

\begin{remark}
By Lemma \ref{lemmaZ} (iv) and Lemma \ref{lemmaZ'} (i), $\int_{\mathbb{Q}%
_{p}^{n}}\frac{\partial Z(x-y,t-\theta)}{\partial t}d^{n}y=0$, then
\[
\frac{\partial u_{2}}{\partial t}(x,t,\tau)=f(x,t)+%
{\displaystyle\int\limits_{\tau}^{t}}
\left(  \underset{\mathbb{Q}_{p}^{n}}{\int}\frac{\partial Z(x-y,t-\theta
)}{\partial t}f(y,\theta)d^{n}y\right)  d\theta\text{,}%
\]
for $t>0$ and$\ x\in\mathbb{Q}_{p}^{n}$.
\end{remark}

\section{\label{Sect4}Parabolic-type \ equations with variable coefficients}

First, we fix the notation that will be used through this section. We fix
$N+1$ positive real numbers satisfying $n<\alpha_{1}<\alpha_{2}<\cdots
<\alpha_{N}<\alpha$. We fix $N+2$ functions $a_{k}(x,t),$ $k=0,\ldots N$ and
$b(x,t)$ from $\mathbb{Q}_{p}^{n}\times\lbrack0,T]$ to $\mathbb{R}$, here $T$
is a positive constant. We assume that: (i) $b(x,t)$ and $a_{k}(x,t)$, for
$k=0,\ldots,N$,\ belong (with respect to $x$) to$\ \mathcal{\mathfrak{M}}_{0}$
uniformly with respect to $t\in\left[  0,T\right]  $; (ii) $a_{0}%
(x,t)$\ satisfies the Hölder condition in $t$ with exponent $v\in(0,1)$
uniformly in $x$. We also assume the uniform parabolicity condition
$a_{0}(x,t)\geq\mu>0$ and that $\alpha_{N+1}:=n+\left(  \alpha-n\right)
(1-v)>\alpha_{N}$.\ Notice that $\alpha_{N+1}<\alpha$.

Set $\widetilde{\boldsymbol{W}}:=\sum_{k=1}^{N}a_{k}%
(x,t)\boldsymbol{\mathbf{W}}_{\alpha_{k}}-b(x,t)\boldsymbol{I}$ with domain
$\mathcal{\mathfrak{M}}_{\lambda}$, and $0\leq\lambda+n<\alpha_{1}$. Notice
that $\widetilde{\boldsymbol{W}}:\mathcal{\mathfrak{M}}_{\lambda}%
\rightarrow\mathcal{\mathfrak{M}}_{\lambda}$.

In this section we construct a solution for the following initial value
problem:%
\begin{equation}
\left\{
\begin{array}
[c]{l}%
\frac{\partial u}{\partial t}(x,t)-a_{0}(x,t)(\boldsymbol{\mathbf{W}}_{\alpha
}u)(x,t)-(\widetilde{\boldsymbol{W}}u)(x,t)=f(x,t)\\
\\
u\left(  x,0\right)  =\varphi(x),
\end{array}
\right.  \label{Cauchy2}%
\end{equation}
where\ $x\in$ $\mathbb{Q}_{p}^{n},$ $t\in(0,T],$ $\varphi(x)\in
\mathcal{\mathfrak{M}}_{\lambda}$, $f(x,t)\in\mathcal{\mathfrak{M}}_{\lambda}$
uniformly with respect to $t$, with $\ 0\leq\lambda<\alpha_{1}-n$, and
$\ f(x,t)$ is continuous in $(x,t)$ (if $a_{1}(x,t)=\cdots=a_{N}(x,t)\equiv0$
then we shall assume that $0\leq\lambda<\alpha-n$).

\subsection{Parametrized Cauchy problem}

We first study the following Cauchy problem:%
\begin{equation}
\left\{
\begin{array}
[c]{l}%
\frac{\partial u}{\partial t}(x,t)-a_{0}(y,\theta)(\boldsymbol{\mathbf{W}%
}_{\alpha}u)(x,t)=0\text{, }\ x\in\mathbb{Q}_{p}^{n},t\in(0,T]\\
\\
u\left(  x,0\right)  =\varphi(x),
\end{array}
\right.  \label{Cauchy3}%
\end{equation}
where $y\in\mathbb{Q}_{p}^{n}$,\ $\theta>0$ are parameters. By taking
$\kappa=a_{0}(y,\theta)\geq\mu>0$ and applying the results of Section
\ref{Sect3}, Cauchy problem (\ref{Cauchy3}) has a fundamental solution \ given
by%
\[
Z(x,t;y,\theta,w_{\alpha},\kappa):=Z(x,t;y,\theta)=%
{\displaystyle\int\limits_{\mathbb{Q}_{p}^{n}}}
\Psi(x\cdot\xi)e^{-a_{0}(y,\theta)tA_{w_{\alpha}}(\left\Vert \xi\right\Vert
_{p})}d^{n}\xi,
\]
for $t>0$ and \ $x\in\mathbb{Q}_{p}^{n}$.

\begin{remark}
All statements from the Lemmas \ref{lemmaZ}, \ref{lemmaZ'}, \ref{lemmaWZ}%
\ hold for $Z(x,t;y,\theta)$ and the involved constants do not depend of $y$
and $\theta$. Thus, we have the following estimates:%
\begin{equation}
Z(x,t;y,\theta)\leq C_{1}t\left(  \left\Vert x\right\Vert _{p}+t^{\frac
{1}{\alpha-n}}\right)  ^{-\alpha}\text{, for }t>0; \label{Est1}%
\end{equation}

\end{remark}%

\begin{equation}
\left\vert \frac{\partial Z(x,t;y,\theta)}{\partial t}\right\vert \leq
C_{2}\left(  \left\Vert x\right\Vert _{p}+t^{\frac{1}{\alpha-n}}\right)
^{-\alpha}\text{, for }t>0; \label{Est2}%
\end{equation}

\begin{equation}
\left\vert \left(  \boldsymbol{W}_{\gamma}Z\right)  (x,t;y,\theta)\right\vert
\leq C_{3}\left(  \left\Vert x\right\Vert _{p}+t^{\frac{1}{\alpha-n}}\right)
^{-\gamma}\text{, for }t>0\text{ and }\gamma\leq\alpha. \label{Est3}%
\end{equation}

And the identities:%
\begin{equation}
\underset{\mathbb{Q}_{p}^{n}}{\int}Z(x,t;y,\theta)d^{n}x=1\text{, for }t>0;
\label{Est4}%
\end{equation}

\begin{equation}
\frac{\partial Z(x,t;y,\theta)}{\partial t}=-a_{0}\left(  y,\theta\right)
\underset{\mathbb{Q}_{p}^{n}}{\int}A_{w_{\alpha}}(\left\Vert \xi\right\Vert
_{p})e^{-a_{0}\left(  y,\theta\right)  tA_{w_{\alpha}}(\left\Vert
\xi\right\Vert _{p})}\Psi(x\cdot\xi)d^{n}\xi\text{, for }t>0; \label{Est5}%
\end{equation}

\begin{equation}
\left(  \mathbf{W}_{\gamma}Z\right)  (x,t;y,\theta)=-\underset{\mathbb{Q}%
_{p}^{n}}{\int}A_{w_{\gamma}}(\left\Vert \xi\right\Vert _{p})e^{-a_{0}\left(
y,\theta\right)  tA_{w_{\alpha}}(\left\Vert \xi\right\Vert _{p})}\Psi
(x\cdot\xi)d^{n}\xi\text{, } \label{Est6}%
\end{equation}
for $t>0$ and $\gamma\leq\alpha$;%

\begin{equation}
\underset{\mathbb{Q}_{p}^{n}}{\int}\left(  \boldsymbol{W}_{\gamma}%
Z_{t}\right)  (x,t;y,\theta)d^{n}x=0. \label{Est7}%
\end{equation}

\begin{lemma}
There exists a positive constant $C$, such that%
\begin{equation}
\left\vert \underset{\mathbb{Q}_{p}^{n}}{\int}\frac{\partial Z(x-y;t,y,\theta
)}{\partial t}d^{n}y\right\vert \leq C. \label{Est8}%
\end{equation}

\end{lemma}

\begin{proof}
See proof Lemma 4.5 in \cite{Koch}.
\end{proof}

\subsection{Heat potentials}

We define \textit{the parameterized heat potentials} as follows:%
\[
u(x,t,\tau):=%
{\displaystyle\int\limits_{\tau}^{t}}
\underset{\mathbb{Q}_{p}^{n}}{\int}Z(x-y,t-\theta;y,\theta)f(y,\theta
)d^{n}yd\theta,
\]
where $f\in\mathcal{\mathfrak{M}}_{\lambda}$, $0\leq\lambda<\alpha-n$, $f$
continuous in $(y,\theta)$. By using the same argument given to prove Lemma
\ref{u1u2}, one proves that $u\in\mathcal{\mathfrak{M}}_{\lambda}$ uniformly
in $t$ and $\tau$.

We now calculate the derivative with respect to\ $t$ and the action of the
operator $\mathbf{W}_{\gamma}$ on $u(x,t,\tau)$\ for $n+\lambda<\gamma
\leq\alpha$.

\begin{proposition}
\label{u3} Assume\ that $f\in\mathcal{\mathfrak{M}}_{\lambda}$, $0\leq
\lambda<\alpha-n$, $f$ continuous in $(y,\theta)$. Then the following
assertions hold:

(i) $\frac{\partial u(x,t,\tau)}{\partial t}=f(x,t)+\int_{\tau}^{t}%
\underset{\mathbb{Q}_{p}^{n}}{\int}\frac{\partial Z(x-y,t-\theta;y,\theta
)}{\partial t}f(y,\theta)d^{n}yd\theta$;

(ii)\ $(\mathbf{W}_{\gamma}u)(x,t,\tau)=\int_{\tau}^{t}\underset
{\mathbb{Q}_{p}^{n}}{\int}(\mathbf{W}_{\gamma}Z)(x-y,t-\theta;y,\theta
)f(y,\theta)d^{n}yd\theta$, $\gamma\leq\alpha$.
\end{proposition}

\begin{proof}
It is\ a simple variation of the proof given for Proposition \ref{u2}.
\end{proof}

The following technical result will be used later on.

\begin{lemma}
[\cite{Koch}, Lemma 4.6]\label{lema J<} Let%
\begin{multline*}
J(x,\xi,t,\tau)=%
{\displaystyle\int\limits_{\tau}^{t}}
(t-\theta)^{-\rho/\beta}(\theta-\tau)^{-\sigma/\beta}\\
\times\left(  \underset{\mathbb{Q}_{p}^{n}}{\int}\left(  (t-\theta)^{1/\beta
}+\left\Vert x-\eta\right\Vert _{p}\right)  ^{-n-b_{1}}\left(  (\theta
-\tau)^{1/\beta}+\left\Vert \eta-\xi\right\Vert _{p}\right)  ^{-n-b_{2}}%
d^{n}\eta\right)  d\theta,
\end{multline*}
where $x$, $\xi\in\mathbb{Q}_{p}^{n}$,$\ 0\leq\tau<t$,\ $b_{1}$, $b_{2}>0$,
$\rho+b_{1}<\beta$, $\sigma+b_{2}<\beta$. Then%
\begin{multline*}
J(x,\xi,t,\tau)\leq\\
C\left\{  B\left(  1-\frac{\rho}{\beta},1-\frac{\sigma+b_{2}}{\beta}\right)
\left(  (t-\tau)^{1/\beta}+\left\Vert x-\xi\right\Vert _{p}\right)
^{-n-b_{1}}(t-\tau)^{-\frac{(\rho+\sigma+b_{2}-\beta)}{\beta}}\right. \\
+\left.  B\left(  1-\frac{\rho+b_{1}}{\beta},1-\frac{\sigma}{\beta}\right)
\left(  (t-\tau)^{1/\beta}+\left\Vert x-\xi\right\Vert _{p}\right)
^{-n-b_{2}}(t-\tau)^{-\frac{(\rho+\sigma+b_{1}-\beta)}{\beta}}\right\}  ,
\end{multline*}
where$\ C$ is a positive constant depends only on $b_{1,}b_{2}$\ and
$B(\cdot,\cdot)$\ denotes the Ar\-chimedean Beta function.
\end{lemma}

The proof is a simple variation of that given by\ Kochubei for Lemma 4.6\ in
\cite{Koch}.

\subsection{Construction of a solution}

\begin{theorem}
\label{mainTheo}The Cauchy problem (\ref{Cauchy2}), has \ a solution, which
can be represented in the form%
\begin{equation}
u(x,t)=%
{\displaystyle\int\limits_{0}^{t}}
{\displaystyle\int\limits_{\mathbb{Q}_{p}^{n}}}
{\LARGE \Lambda}(x,t,\xi,\tau)f(\xi,\tau)d^{n}\xi d\tau+%
{\displaystyle\int\limits_{\mathbb{Q}_{p}^{n}}}
{\LARGE \Lambda}(x,t,\xi,0)\varphi(\xi)d^{n}\xi, \label{u(x,t)}%
\end{equation}
where the fundamental solution ${\LARGE \Lambda}(x,t,\xi,\tau)$, $x$, $\xi
\in\mathbb{Q}_{p}^{n}$, $0\leq\tau<t\leq T$, has \ the form%
\begin{equation}
{\LARGE \Lambda}(x,t,\xi,\tau)=Z(x-\xi,t-\tau;\xi,\tau)+\mathcal{W}\left(
x,t,\xi,\tau\right)  , \label{Gamma}%
\end{equation}
with%
\begin{align}
\left\vert \mathcal{W}\left(  x,t,\xi,\tau\right)  \right\vert  &  \leq
C\left\{  (t-\tau)^{2-\frac{\lambda}{\alpha-n}}\left[  (t-\tau)^{\frac
{1}{\alpha-n}}+\left\Vert x-\xi\right\Vert _{p}\right]  ^{-\alpha}\right.
\label{Cotateo}\\
&  +\left.  (t-\tau)%
{\displaystyle\sum\limits_{k=1}^{N+1}}
\left[  (t-\tau)^{\frac{1}{\alpha-n}}+\left\Vert x-\xi\right\Vert _{p}\right]
^{-\alpha_{k}}\right\}  .\nonumber
\end{align}

Furthermore $Z(x,t;y,\theta)$ satisfies the estimates (\ref{Est1}),
(\ref{Est2}), (\ref{Est3}), (\ref{Est8}).
\end{theorem}

\begin{proof}
We use the usual scheme of Levi's method. Thus, we look for a fundamental
solution of (\ref{Cauchy2}) having\ form (\ref{Gamma}), with
\[
\mathcal{W}\left(  x,t,\xi,\tau\right)  =%
{\displaystyle\int\limits_{\tau}^{t}}
\underset{\mathbb{Q}_{p}^{n}}{\int}Z(x-\eta,t-\theta;\eta,\theta)\Phi
(\eta,\theta,\xi,\tau)d^{n}\eta d\theta.
\]
Then,\ we require that%
\begin{gather}
\frac{\partial{\LARGE \Lambda}}{\partial t}(x,t,\xi,\tau)-a_{0}%
(x,t)(\boldsymbol{\mathbf{W}}_{\alpha}{\LARGE \Lambda})(x,t,\xi,\tau
)-\sum_{k=1}^{N}a_{k}(x,t)(\boldsymbol{\mathbf{W}}_{\alpha_{k}}{\LARGE \Lambda
})(x,t,\xi,\tau)\nonumber\\
+b(x,t){\LARGE \Lambda}(x,t,\xi,\tau)=0\text{,} \label{Gamma1}%
\end{gather}
for $x\neq0$, $t>0$. Now by using (\ref{Gamma}), (\ref{Est5})-(\ref{Est6}) and
Proposition \ref{u3}, we have formally%
\begin{gather*}
\frac{\partial Z}{\partial t}(x-\xi,t-\tau,\xi,\tau)+\Phi(x,t,\xi,\tau)+%
{\displaystyle\int\limits_{\tau}^{t}}
\underset{\mathbb{Q}_{p}^{n}}{\int}\frac{\partial Z(x-\eta,t-\theta
;\eta,\theta)}{\partial t}\Phi(\eta,\theta,\xi,\tau)d^{n}\eta d\theta\\
-a_{0}(x,t)\bigl\{(\boldsymbol{\mathbf{W}}_{\alpha}Z)(x-\xi,t-\tau;\xi,\tau)+%
{\displaystyle\int\limits_{\tau}^{t}}
\underset{\mathbb{Q}_{p}^{n}}{\int}(\boldsymbol{\mathbf{W}}_{\alpha}%
Z)(x-\eta,t-\theta;\eta,\theta)\times\\
\Phi(\eta,\theta,\xi,\tau)d^{n}\eta d\theta\bigr\}-%
{\displaystyle\sum\limits_{k=1}^{N}}
a_{k}(x,t)\bigl\{(\boldsymbol{\mathbf{W}}_{\alpha_{k}}Z)(x-\xi,t-\tau;\xi
,\tau)\\
+%
{\displaystyle\int\limits_{\tau}^{t}}
\underset{\mathbb{Q}_{p}^{n}}{\int}(\boldsymbol{\mathbf{W}}_{\alpha_{k}%
}Z)(x-\eta,t-\theta;\eta,\theta)\Phi(\eta,\theta,\xi,\tau)d^{n}\eta
d\theta\bigr\}\\
+b(x,t)\bigl\{Z(x-\xi,t-\tau;\xi,\tau)+%
{\displaystyle\int\limits_{\tau}^{t}}
\underset{\mathbb{Q}_{p}^{n}}{\int}Z(x-\eta,t-\theta;\eta,\theta)\Phi
(\eta,\theta,\xi,\tau)d^{n}\eta d\theta\bigr\}=0.
\end{gather*}
By taking%
\begin{align*}
R(x,t,\xi,\tau)  &  :=(a_{0}(x,t)-a_{0}(\xi,\tau))(\mathbf{W}_{\alpha}%
Z)(x-\xi,t-\tau;\xi,\tau)\\
&  +\sum_{k=1}^{N}a_{k}(x,t)(\boldsymbol{\mathbf{W}}_{\alpha_{k}}%
Z)(x-\xi,t-\tau;\xi,\tau)-b(x,t)Z(x-\xi,t-\tau;\xi,\tau),
\end{align*}
one gets that\ $\Phi(x,t,\xi,\tau)$ satisfies the integral equation%
\begin{equation}
\Phi(x,t,\xi,\tau)=R\ (x,t,\xi,\tau)+%
{\displaystyle\int\limits_{\tau}^{t}}
\underset{\mathbb{Q}_{p}^{n}}{\int}R(x,t,\eta,\theta)\Phi(\eta,\theta,\xi
,\tau)d^{n}\eta d\theta. \label{IntEqua}%
\end{equation}
Now, by using (\ref{Est3}) and (\ref{Est1}), we obtain%
\begin{align}
\left\vert R(x,t,\xi,\tau)\right\vert  &  \leq C_{0}\left(  \left\vert
a_{0}(x,t)-a_{0}(\xi,\tau)\right\vert ((t-\tau)^{\frac{1}{\alpha-n}%
}+\left\Vert x-\xi\right\Vert _{p})^{-\alpha}\right. \nonumber\\
&  +%
{\displaystyle\sum\limits_{k=1}^{N}}
\left(  (t-\tau)^{\frac{1}{\alpha-n}}+\left\Vert x-\xi\right\Vert _{p}\right)
^{-\alpha_{k}}\nonumber\\
&  \left.  +\left(  (t-\tau)^{\frac{1}{\alpha-n}}+\left\Vert x-\xi\right\Vert
_{p}\right)  ^{-\alpha}(t-\tau)\right)  . \label{Est_R}%
\end{align}

\textbf{Claim A. }%
\[
\left\vert a_{0}(x,t)-a_{0}(\xi,\tau)\right\vert ((t-\tau)^{\frac{1}{\alpha
-n}}+\left\Vert x-\xi\right\Vert _{p})^{-\alpha}\leq C_{1}^{\prime}%
((t-\tau)^{\frac{1}{\alpha-n}}+\left\Vert x-\xi\right\Vert _{p})^{-\alpha
_{N+1}},
\]
where $\alpha_{N+1}=n+\left(  \alpha-n\right)  (1-v)>\alpha_{N}$.

Indeed, by the Hölder condition\
\[
\left\vert a_{0}(x,t)-a_{0}(\xi,\tau)\right\vert \leq C_{1}(t-\tau
)^{v}+\left\vert a_{0}\left(  x,\tau\right)  -a_{0}\left(  \xi,\tau\right)
\right\vert .
\]

Let $l\left(  a_{0}\right)  $ be the parameter of local constancy of $a_{0}$.
Thus, if $\left\Vert x-\xi\right\Vert _{p}\leq p^{l\left(  a_{0}\right)  }$,
then $\left\vert a_{0}(x,t)-a_{0}(\xi,\tau)\right\vert \leq C_{1}(t-\tau)^{v}%
$. In the case $\left\Vert x-\xi\right\Vert _{p}\leq p^{l\left(  a_{0}\right)
}$, the inequality follows from the fact that $(t-\tau)^{v}((t-\tau)^{\frac
{1}{\alpha-n}}+\left\Vert x-\xi\right\Vert _{p})^{-\alpha+\alpha_{N+1}}$ is
bounded, which in turn follows from $\lim_{t\rightarrow\tau}(t-\tau
)^{v+\frac{-\alpha+\alpha_{N+1}}{\alpha-n}}=1$. In the case $\left\Vert
x-\xi\right\Vert _{p}>p^{l\left(  a_{0}\right)  }$, taking $\left\vert
a_{0}(x,t)-a_{0}(\xi,\tau)\right\vert \leq C_{0}$, \ the inequality follows
from%
\[
((t-\tau)^{\frac{1}{\alpha-n}}+\left\Vert x-\xi\right\Vert _{p})^{-\alpha
+\alpha_{N+1}}\leq\left\Vert x-\xi\right\Vert _{p}{}^{-\alpha+\alpha_{N+1}%
}\leq p{}^{\left(  -\alpha+\alpha_{N+1}\right)  l\left(  a_{0}\right)  }.
\]

\textbf{Claim B.}
\[
\left(  t-\tau\right)  ((t-\tau)^{\frac{1}{\alpha-n}}+\left\Vert
x-\xi\right\Vert _{p})^{-\alpha}\leq C_{2}((t-\tau)^{\frac{1}{\alpha-n}%
}+\left\Vert x-\xi\right\Vert _{p})^{-\alpha_{N+1}}.
\]

This assertion is a consequence of the fact that $\lim_{t\rightarrow\tau
}(t-\tau)^{1+\frac{-\alpha+\alpha_{N+1}}{\alpha-n}}=0$.

Now from (\ref{Est_R}), and Claims A-B, we have%
\begin{equation}
\left\vert R(x,t,\xi,\tau)\right\vert \leq C%
{\displaystyle\sum\limits_{k=1}^{N+1}}
\left[  (t-\tau)^{\frac{1}{\alpha-n}}+\left\Vert x-\xi\right\Vert _{p}\right]
^{-\alpha_{k}}. \label{R<}%
\end{equation}
We solve\ integral equation (\ref{IntEqua}) by the method of successive
approximations:%
\begin{equation}
\Phi(x,t,\xi,\tau)=%
{\displaystyle\sum\limits_{m=1}^{\infty}}
R_{m}(x,t,\eta,\theta), \label{PHI}%
\end{equation}
where$\ R_{1}\equiv R$ and%
\[
R_{m+1}(x,t,\xi,\tau)=%
{\displaystyle\int\limits_{\tau}^{t}}
\underset{\mathbb{Q}_{p}^{n}}{\int}R(x,t,\eta,\theta)R_{m}(\eta,\theta
,\xi,\tau)d^{n}\eta d\theta\text{, for\ }m\geq1.
\]

\textbf{Claim C.}%
\begin{align*}
\left\vert R_{m+1}(x,t,\xi,\tau)\right\vert  &  \leq C(2N+2)^{m}(t-\tau
)^{mv}\frac{\left(  \Gamma(v)\right)  ^{m+1}}{\Gamma((m+1)v)}\times\\
&
{\displaystyle\sum\limits_{j=1}^{N+1}}
\left[  (t-\tau)^{\frac{1}{\alpha-n}}+\left\Vert x-\xi\right\Vert _{p}\right]
^{-\alpha_{j}},
\end{align*}
for $m\geq0$, where $\Gamma\left(  \cdot\right)  $ denotes the Archimedean
Gamma function.

The proof of this assertion will be given later.

It follows from Claim A, by the Stirling formula, that series (\ref{PHI}) is
convergent and that%
\begin{equation}
\left\vert \Phi(x,t,\xi,\tau)\right\vert \leq C_{0}%
{\displaystyle\sum\limits_{k=1}^{N+1}}
\left[  (t-\tau)^{\frac{1}{\alpha-n}}+\left\Vert x-\xi\right\Vert _{p}\right]
^{-\alpha_{k}}. \label{phi<}%
\end{equation}
Now (\ref{Cotateo}) follows from (\ref{phi<}) \ and\ Lemma \ref{lema J<}.

Denote by\ $u_{1}(x,t)$ \ and\ $u_{2}(x,t)$ \ the first and second terms in
the right hand side of (\ref{u(x,t)}). Substituting \ (\ref{Gamma}) into
(\ref{u(x,t)}), we find that%
\[
u_{1}(x,t)=%
{\displaystyle\int\limits_{0}^{t}}
{\displaystyle\int\limits_{\mathbb{Q}_{p}^{n}}}
Z(x-\xi,t-\tau;\xi,\tau)f(\xi,\tau)d^{n}\xi d\tau+%
{\displaystyle\int\limits_{0}^{t}}
{\displaystyle\int\limits_{\mathbb{Q}_{p}^{n}}}
Z(x-\eta,t-\theta;\eta,\theta)F(\eta,\theta)d^{n}\eta d\theta,
\]
and%
\[
u_{2}(x,t)=%
{\displaystyle\int\limits_{\mathbb{Q}_{p}^{n}}}
Z(x-\xi,t;\xi,0)\varphi(\xi)d^{n}\xi+%
{\displaystyle\int\limits_{0}^{t}}
{\displaystyle\int\limits_{\mathbb{Q}_{p}^{n}}}
Z(x-\eta,t-\theta;\eta,\theta)G(\eta,\theta)d^{n}\eta d\theta,
\]
where%
\begin{align}
F(\eta,\theta)  &  =%
{\displaystyle\int\limits_{0}^{\theta}}
{\displaystyle\int\limits_{\mathbb{Q}_{p}^{n}}}
\Phi(\eta,\theta,\xi,\tau)f(\xi,\tau)d^{n}\xi d\tau,\label{FG0}\\
G(\eta,\theta)  &  =%
{\displaystyle\int\limits_{\mathbb{Q}_{p}^{n}}}
\Phi(\eta,\theta,\xi,0)\varphi(\xi)d^{n}\xi. \label{FG}%
\end{align}
Now, by Proposition \ref{Conv} and\ (\ref{phi<}), it \ follows that%
\[
\left\vert F(\eta,\theta)\right\vert \leq C_{0}(1+\left\Vert \eta\right\Vert
_{p}^{\lambda})\text{, \ \ \ }\left\vert G(\eta,\theta)\right\vert \leq
C_{1}(1+\left\Vert \eta\right\Vert _{p}^{\lambda}),
\]
for all $\ \eta\in\mathbb{Q}_{p}^{n}$ and \ $\theta\in(0,T]$.

\textbf{Claim D.} The functions \ $F$ and $G$ belong to $\widetilde
{\mathcal{E}\text{,}}$ and their parameters of local constancy do not depend
on $\theta$.

We first note that by (\ref{FG0})-(\ref{FG}), it is sufficient to show that
$\Phi(\cdot,\theta,\star,\tau)$ is a locally constant function on $\left(
\mathbb{Q}_{p}^{\times}\right)  ^{n}\times\mathbb{Q}_{p}^{n}$ and that its
parameter of local constancy do not depend on $\theta$ and $\tau$. Now, by the
recursive definition of the function $\Phi$ we see that if $\ L$ is the
parameter local constancy for all the functions $a_{k}(\cdot,t)$, $b(\cdot
,t)$, $(\boldsymbol{\mathbf{W}}_{\alpha_{k}}Z)(\cdot,t-\tau;\star,\tau)$ and
$Z(\cdot,t-\tau;\star,\tau)$ on $\left(  \mathbb{Q}_{p}^{\times}\right)
^{n}\times\mathbb{Q}_{p}^{n}$, and\ if $\ \left\Vert \delta\right\Vert
_{p}\leq p^{-L}$,\ we have%
\[
R(x+\delta,t,\xi+\delta,\tau)=R(x,t,\xi,\tau).
\]
Furthermore, we successively obtain%
\begin{align*}
R_{m+1}(x+\delta,t,\xi+\delta,\tau)  &  =%
{\displaystyle\int\limits_{\tau}^{t}}
\underset{\mathbb{Q}_{p}^{n}}{\int}R(x+\delta,t,\eta,\theta)R_{m}(\eta
,\theta,\xi+\delta,\tau)d^{n}\eta d\theta\\
&  =%
{\displaystyle\int\limits_{\tau}^{t}}
\underset{\mathbb{Q}_{p}^{n}}{\int}R(x+\delta,t,\zeta+\delta,\theta
)R_{m}(\zeta+\delta,\theta,\xi+\delta,\tau)d^{n}\zeta d\theta\\
&  =R_{m+1}(x,t,\xi,\tau),
\end{align*}
so that \ $\Phi(x+\delta,t,\xi+\delta,\tau)=\Phi(x,t,\xi,\tau)$,\ and hence%
\begin{align*}
F(\eta+\delta,\theta)  &  =%
{\displaystyle\int\limits_{0}^{\theta}}
{\displaystyle\int\limits_{\mathbb{Q}_{p}^{n}}}
\Phi(\eta+\delta,\theta,\xi,\tau)f(\xi,\tau)d^{n}\xi d\tau\\
&  =%
{\displaystyle\int\limits_{0}^{\theta}}
{\displaystyle\int\limits_{\mathbb{Q}_{p}^{n}}}
\Phi(\eta+\delta,\theta,\xi+\delta,\tau)f(\xi+\delta,\tau)d^{n}\xi
d\tau=F(\eta,\theta).
\end{align*}
Similarly, $G(\eta+\delta,\theta)=G(\eta,\theta)$ when \ $\left\vert
\delta\right\vert _{p}\leq p^{-L}$. Thus $u_{1}(x,t)$, $u_{2}(x,t)\in
\mathcal{\mathfrak{M}}_{\lambda}\ $uniformly in $t$. Thus the potentials in
the expressions for $u_{1}(x,t)$, $u_{2}(x,t)$ satisfy the conditions to use
the differentiation formulas given in Proposition \ref{u3}. By using these
formulas along with Proposition \ref{u3}, (\ref{Est5})-(\ref{Est6}) and
(\ref{IntEqua}), ones verifies after simple transformations that $u(x,t)$ is a
solution of Cauchy problem \ (\ref{Cauchy2}).

Let us show that $u(x,t)\rightarrow\varphi(x)$ as $t\rightarrow0^{+}$. Due to
(\ref{Gamma}) and (\ref{Cotateo}), it \ is sufficient to verify that
\[
v(x,t):=%
{\displaystyle\int\limits_{\mathbb{Q}_{p}^{n}}}
Z(x-\xi,t;\xi,0)\varphi(\xi)d^{n}\xi\rightarrow\varphi(x)\text{ \ as
\ }t\rightarrow0^{+}.\
\]
By virtue of formula (\ref{Est4}), we have
\begin{align*}
v(x,t)  &  =%
{\displaystyle\int\limits_{\mathbb{Q}_{p}^{n}}}
\left[  Z(x-\xi,t;\xi,0)-Z(x-\xi,t;x,0)\right]  \varphi(\xi)d^{n}\xi\\
&  +%
{\displaystyle\int\limits_{\mathbb{Q}_{p}^{n}}}
Z(x-\xi,t;x,0)\left[  \varphi(\xi)-\varphi(x)\right]  d^{n}\xi+\varphi(x).
\end{align*}
Now, since $Z(x-\xi,t;\cdot,0)$\ and $\varphi(\cdot)$ are locally constant
functions, it follows that in both integrals the integration is actually
performed over the set%
\[
\left\{  \xi\in\mathbb{Q}_{p}^{n}:\left\Vert \xi-x\right\Vert _{p}\geq
p^{-L}\right\}  .
\]

By applying (\ref{Est1}) on this set, we see that both integrals tend to zero
as $t\rightarrow0^{+}$.

\textbf{Proof \ of Claim C. }We use induction on $m$. The case $m=0$ is
(\ref{R<}). We assume the case $m$,\ then%
\begin{align*}
\left\vert R_{m+1}(x,t,\xi,\tau)\right\vert  &  \leq%
{\displaystyle\int\limits_{\tau}^{t}}
\underset{\mathbb{Q}_{p}^{n}}{\int}\left\vert R(x,t,\eta,\theta)\right\vert
\left\vert R_{m}(\eta,\theta,\xi,\tau)\right\vert d^{n}\eta d\theta\\
&  \leq C_{0}(2N+2)^{m-1}\frac{\left(  \Gamma(v)\right)  ^{m}}{\Gamma(mv)}%
{\displaystyle\sum\limits_{j,k=1}^{N+1}}
\int_{\tau}^{t}(\theta-\tau)^{(m-1)v}\times\\
&  \underset{\mathbb{Q}_{p}^{n}}{\int}\left[  (\theta-\tau)^{\frac{1}%
{\alpha-n}}+\left\Vert \eta-\xi\right\Vert _{p}\right]  ^{-\alpha_{j}}\left[
(t-\theta)^{\frac{1}{\alpha-n}}+\left\Vert x-\eta\right\Vert _{p}\right]
^{-\alpha_{k}}d^{n}\eta d\theta.
\end{align*}
Now by Lemma \ref{lema J<}, with $-\sigma=(m-1)(\alpha-n)v$, $\rho
=0$,\ $-n-b_{2}=-\alpha_{j}$, $-n-b_{1}=-\alpha_{k}$, $\beta=\alpha-n$, we
have%
\begin{align*}
&  \left\vert R_{m+1}(x,t,\xi,\tau)\right\vert \\
&  \leq C_{0}\left\{  (2N+2)^{m-1}\frac{\left(  \Gamma(v)\right)  ^{m}}%
{\Gamma(mv)}%
{\displaystyle\sum\limits_{j,k=1}^{N+1}}
B\left(  1,\frac{\alpha+(m-1)(\alpha-n)v-\alpha_{j}}{\alpha-n}\right)
\times\right. \\
&  \left(  (t-\tau)^{1/(\alpha-n)}+\left\Vert x-\xi\right\Vert _{p}\right)
^{-\alpha_{k}}(t-\tau)^{^{\frac{(m-1)(\alpha-n)v-\alpha_{j}+\alpha)}{\alpha
-n}}}\\
&  +B\left(  \frac{\alpha-\alpha_{k}}{\alpha-n},\frac{\alpha-n+(m-1)(\alpha
-n)v}{\alpha-n}\right)  \left(  (t-\tau)^{1/(\alpha-n)}+\left\Vert
x-\xi\right\Vert _{p}\right)  ^{-\alpha_{j}}\times\\
&  \left.  (t-\tau)^{^{\frac{(m-1)(\alpha-n)v-\alpha_{k}+\alpha)}{\alpha-n}}%
}\right\}  .
\end{align*}

By using $B(z_{1}+\epsilon,z_{2}+\delta)\leq B(z_{1},z_{2})$, for
$\epsilon,\delta\geq0$,%
\begin{align*}
B\left(  1,\frac{\alpha+(m-1)(\alpha-n)v-\alpha_{j}}{\alpha-n}\right)   &
\leq B\left(  v,mv\right)  ,\\
B\left(  \frac{\alpha-\alpha_{k}}{\alpha-n},\frac{\alpha-n+(m-1)(\alpha
-n)v}{\alpha-n}\right)   &  \leq B\left(  v,mv\right)  ,
\end{align*}
and%
\[
(t-\tau)^{^{\frac{(m-1)(\alpha-n)v-\alpha_{r}+\alpha)}{\alpha-n}}}%
=(t-\tau)^{^{mv-\frac{(\alpha-n)v+\alpha_{r}-\alpha}{\alpha-n}}}\leq
C(t-\tau)^{^{mv}},
\]
for $1\leq r\leq N+1$, we get%
\begin{align*}
\left\vert R_{m+1}(x,t,\xi,\tau)\right\vert  &  \leq C(2N+2)^{m}\frac{\left(
\Gamma(v)\right)  ^{m}}{\Gamma(\left(  m+1\right)  v)}(t-\tau)^{^{mv}}\\
&  \times%
{\displaystyle\sum\limits_{k=1}^{N+1}}
\left(  (t-\tau)^{1/(\alpha-n)}+\left\Vert x-\xi\right\Vert _{p}\right)
^{-\alpha_{k}}.
\end{align*}

\end{proof}

\section{\label{sect5}Uniqueness of the Solution}

We recall that $\widetilde{\mathcal{E}}$ is the $\mathbb{C}$-vector space of
all functions $\varphi:\mathbb{Q}_{p}^{n}\rightarrow\mathbb{C}$, such that
there exist a ball $B_{l}^{n}$, with $l$ depending only on $\varphi$, and
$\varphi\left(  x+x^{\prime}\right)  =\varphi\left(  x\right)  $ for any
$x^{\prime}\in B_{l}^{n}$. Notice that $\mathfrak{M}_{\lambda}\subset
\widetilde{\mathcal{E}}$ for any $\lambda$. We identify any element of
$\widetilde{\mathcal{E}}$\ with a distribution on $\mathbb{Q}_{p}^{n}$. We now
recall the following fact:$\ T\in S^{\prime}$ with $\ supp(T)\subset B_{N}%
^{n}$ if only if $\widehat{T}$ $\in\widetilde{\mathcal{E}}$ and its parameter
of local constancy is greater than $-N$, cf. \cite[pg 109]{V-V-Z} or
\cite[Proposition 3.17]{Taibleson}.

\begin{lemma}
\label{NewDomain} $\mathbf{W}_{\alpha}:\widetilde{\mathcal{E}}\rightarrow
\widetilde{\mathcal{E}}$ is a well-defined linear operator. Furthermore,%
\[
(\mathbf{W}_{\alpha}\varphi)(x)=-\mathcal{F}_{\xi\rightarrow x}^{-1}\left(
A_{w_{\alpha}}(\left\Vert \xi\right\Vert _{p})\mathcal{F}_{x\rightarrow\xi
}\varphi\right)  .
\]

\end{lemma}

\begin{proof}
Let $l$ be a parameter of \ locally constancy of $\varphi$, then
\begin{align*}
(\mathbf{W}_{\alpha}\varphi)(x)  &  =%
{\displaystyle\int\limits_{\left\Vert y\right\Vert _{p}\geq p^{l}}}
\frac{\varphi(x-y)-\varphi(x)}{w_{\alpha}(\left\Vert y\right\Vert _{p})}%
d^{n}y\\
&  =\frac{{\LARGE 1}_{\mathbb{Q}_{p}^{n}\smallsetminus B_{l}^{n}}\left(
x\right)  }{w_{\alpha}\left(  \left\Vert x\right\Vert _{p}\right)  }%
\ast\varphi\left(  x\right)  -\varphi\left(  x\right)  \left(  {\int
\limits_{\left\Vert y\right\Vert _{p}\geq p^{l}}}\frac{d^{n}y}{w_{\alpha
}(\left\Vert y\right\Vert _{p})}\right)  .
\end{align*}
Then by taking the Fourier transform in $S^{\prime}$:
\[
\mathcal{F}(\mathbf{W}_{\alpha}\varphi)(\xi)=\left(
{\displaystyle\int\limits_{\mathbb{Q}_{p}^{n}}}
{\LARGE 1}_{\mathbb{Q}_{p}^{n}\smallsetminus B_{l}^{n}}\left(  x\right)
\frac{(\Psi(x\cdot\xi)-1)}{w_{\alpha}(\left\Vert x\right\Vert _{p})}%
d^{n}x\right)  \left(  \mathcal{F}\varphi\right)  (\xi),
\]
and since $\mathcal{F}\varphi\in S^{\prime}$ with $supp(\mathcal{F}%
\varphi)\subset B_{-l}^{n}$,%
\[
\mathcal{F}(\mathbf{W}_{\alpha}\varphi)(\xi)=\left(
{\displaystyle\int\limits_{\mathbb{Q}_{p}^{n}}}
\frac{(\Psi(x\cdot\xi)-1)}{w_{\alpha}(\left\Vert x\right\Vert _{p})}%
d^{n}x\right)  \mathcal{F}\varphi(\xi).
\]
Therefore,%
\[
(\mathbf{W}_{\alpha}\varphi)(x)=-\mathcal{F}_{\xi\rightarrow x}^{-1}\left(
A_{w_{\alpha}}(\left\Vert \xi\right\Vert _{p})\mathcal{F}_{x\rightarrow\xi
}\varphi\right)  \in\widetilde{\mathcal{E}}.
\]

\end{proof}

Take $\gamma$ be a real number such that $\lambda<\gamma<\alpha_{1}%
-n<\ldots<\alpha_{N}-n<\alpha-n$, and fix a integer $L$, and set
$\psi(x):=p^{Ln}\Omega(p^{L}\left\Vert x\right\Vert _{p})\ast\left\Vert
x\right\Vert _{p}^{\gamma}$, then%
\[
\psi(x)=\left\{
\begin{array}
[c]{ll}%
\left\Vert x\right\Vert _{p}^{\gamma} & \text{if \ }\left\Vert x\right\Vert
_{p}>p^{-L}\\
& \\
C & \text{if }\left\Vert x\right\Vert _{p}\leq p^{-L}%
\end{array}
\right.  ,
\]
and thus $\psi\in\widetilde{\mathcal{E}}$.

\begin{lemma}
\label{W(phi)<}With above notation, there exist positive constants $C_{1}$ and
$C_{2}$ such that (i) $\left\vert (\mathbf{W}_{\alpha}\psi)(x)\right\vert \leq
C_{1}\left\Vert x\right\Vert _{p}^{\alpha-\gamma+n}$ and (ii) $\left\vert
(\mathbf{W}_{\alpha_{k}}\psi)(x)\right\vert \leq C_{2}\left\Vert x\right\Vert
_{p}^{\alpha_{k}-\gamma+n}$, for $k=1,\ldots,N$.
\end{lemma}

\begin{proof}
By Lemma \ref{NewDomain},%
\[
(\mathbf{W}_{\alpha}\psi)(x)=-\mathcal{F}_{\xi\rightarrow x}^{-1}\left(
A_{w_{\alpha}}(\left\Vert \xi\right\Vert _{p})\Omega(p^{-L}\left\Vert
\xi\right\Vert _{p})\frac{\left\Vert \xi\right\Vert _{p}^{-\gamma-n}}%
{\Gamma_{n}(n+\gamma)}\right)  \text{\ in }S^{\prime},
\]
where $\Gamma_{n}(n+\gamma)=\frac{1-p^{\gamma}}{1-p^{-\gamma-n}}$. Now, since
$A_{w_{\alpha}}(\left\Vert \xi\right\Vert _{p})\Omega(p^{-L}\left\Vert
\xi\right\Vert _{p})\frac{\left\Vert \xi\right\Vert _{p}^{-\gamma-n}}%
{\Gamma_{n}(n+\gamma)}$ is radial and locally integrable, by applying the
formula for Fourier transform of radial function,%
\begin{multline*}
(\mathbf{W}_{\alpha}\varphi)(x)=\\
\frac{-\left\Vert x\right\Vert _{p}^{-n}}{\Gamma_{n}(n+\gamma)}[(1-p^{-n}%
)\left\Vert x\right\Vert _{p}^{\gamma+n}%
{\displaystyle\sum\limits_{j=0}^{\infty}}
A_{w_{\alpha}}(\left\Vert x\right\Vert _{p}^{-1}p^{-j})\Omega(\left\Vert
x\right\Vert _{p}^{-1-j}p^{-j})p^{j(\gamma+n)-jn}\\
-A_{w_{\alpha}}(\left\Vert x\right\Vert _{p}p^{-j})\Omega(\left\Vert
x\right\Vert _{p}^{-1}p^{-L+1})\left\Vert x\right\Vert _{p}^{\gamma+n}],
\end{multline*}
as a distribution on $\mathbb{Q}_{p}^{n}\smallsetminus\left\{  0\right\}  $,
now by using Lemma 3.4 in \cite{Ch-Z},%
\[
\left\vert (\mathbf{W}_{\alpha}\varphi)(x)\right\vert \leq C^{\prime}\left[
(1-p^{-n})%
{\displaystyle\sum\limits_{j=0}^{\infty}}
p^{-j(\alpha-n)+j\gamma}-p^{(-L+1)(\alpha-n)}\right]  \left\Vert x\right\Vert
_{p}^{-\alpha+n+\gamma}.
\]
The proof of (ii) is similar.
\end{proof}

\begin{theorem}
\label{uniqueTheo}Assume that the coefficients $a_{k}(x,t),$ $k=0,1,\cdots,N$
are nonnegative bounded continuous functions, $b(x,t)$ is a bounded continuous
function, $0\leq\lambda<\alpha_{1}-n$ (if $a_{1}(x,t)=\cdots$ $a_{k}%
(x,t)\equiv0,$ we shall \ suppose that $0\leq\lambda<\alpha-n$) and $u(x,t)$
is a solution of Cauchy problem (\ref{Cauchy2}) with \ $f(x,t)=\varphi
(x)\equiv0$ that belongs to class $\mathcal{\mathfrak{M}}_{\lambda}$. Then
$u(x,t)\equiv0$.
\end{theorem}

\begin{proof}
We may assume that $b(x,t)\geq0$, otherwise we take $u(x,t)e^{\lambda t}$ with
$\lambda>b(x,t)$. We prove that $u(x,t)\geq0$. By contradiction, suppose that
$u(x^{\prime},t^{\prime})<0$, for \ some $x^{\prime}\in\mathbb{Q}_{p}^{n}$ and
$t^{\prime}\in(0,T]$. By Lemma \ref{W(phi)<},\ it follows that $(\mathbf{W}%
_{\alpha}\psi)(x)$ and $(\mathbf{W}_{\alpha_{k}}\psi)(x)\rightarrow0$\ as
$\left\Vert x\right\Vert _{p}\rightarrow\infty$, and thus%
\[
M:=\sup_{\substack{0\leq t\leq T,\\x\in\mathbb{Q}_{p}^{n}}}\left\{
a_{0}(x,t)\left\vert (\boldsymbol{\mathbf{W}}_{\alpha}\psi)(x)\right\vert
+\sum_{k=1}^{N}a_{k}(x,t)\left\vert (\boldsymbol{\mathbf{W}}_{\alpha_{k}}%
\psi)(x)\right\vert \right\}  <\infty.
\]
We pick $\rho>0$ such that $u(x^{\prime},t^{\prime})+T\rho<0$, and then
$\sigma>0$ such that
\begin{equation}
u(x^{\prime},t^{\prime})+T\rho+\sigma\psi(x^{\prime})<0 \label{v<0}%
\end{equation}%
\begin{equation}
\rho-\sigma M<0. \label{p-sigmaM<0}%
\end{equation}
We now consider the function%
\[
v(x,t):=u(x,t)+t\rho+\sigma\psi(x).
\]
From (\ref{v<0}), it follows that $v(x^{\prime},t^{\prime})<0$, so that%
\[
\inf_{\substack{0\leq t\leq T,\\x\in\mathbb{Q}_{p}^{n}}}v\left(  x,t\right)
<0.
\]
On the other hand, since\ $u(x,t)\in\mathcal{\mathfrak{M}}_{\lambda}%
$,\ $\underset{\left\Vert x\right\Vert _{p}\rightarrow\infty}{\lim}%
\frac{u(x,t)}{\psi(x)}=0$ and thus $\underset{\left\Vert x\right\Vert
_{p}\rightarrow\infty}{\lim}v(x,t)>0$ for any $t>0$. This implies that there
exist $\ x_{0}\in\mathbb{Q}_{p}^{n}$ and $t_{0}\in(0,T]$, such that%
\[
\inf_{\substack{0\leq t\leq T,\\x\in\mathbb{Q}_{p}^{n}}}v\left(  x,t\right)
=\min_{\substack{0\leq t\leq T,\\x\in\mathbb{Q}_{p}^{n}}}v(x,t)=v(x_{0}%
,t_{0})<0,
\]
and thus, by formula (\ref{W}), $(\mathbf{W}_{\alpha}v)(x_{0},t_{0}%
)\geq0,(\mathbf{W}_{\alpha_{k}}v)(x_{0},t_{0})\geq0$ for all $k$, and
$\frac{\partial v}{\partial t}(x_{0},t_{0})\leq0$, hence%
\[
\frac{\partial v}{\partial t}(x_{0},t_{0})-a_{0}(x,t)(\boldsymbol{\mathbf{W}%
}_{\alpha}v)(x_{0},t_{0})-\sum_{k=1}^{N}a_{k}(x,t)(\boldsymbol{\mathbf{W}%
}_{\alpha_{k}}v)(x_{0},t_{0})+b(x,t)v(x_{0},t_{0})<0.
\]
Now, by (\ref{p-sigmaM<0}),%
\begin{multline*}
\frac{\partial v}{\partial t}(x,t)-a_{0}(x,t)(\boldsymbol{\mathbf{W}}_{\alpha
}v)(x,t)-\sum_{k=1}^{N}a_{k}(x,t)(\boldsymbol{\mathbf{W}}_{\alpha_{k}%
}v)(x,t)+b(x,t)v(x,t)\\
=\rho-\sigma\left[  a_{0}(x,t)(\boldsymbol{\mathbf{W}}_{\alpha}\psi
)(x)+\sum_{k=1}^{N}a_{k}(x,t)(\boldsymbol{\mathbf{W}}_{\alpha_{k}}%
\psi)(x)\right]  +b(x,t)\left[  \rho t+\sigma\psi(x)\right] \\
\geq\rho-\sigma M>0.
\end{multline*}
We have obtained a contradiction, thus $u(x,t)\geq0$. Finally taking $-u(x,t)$
instead of $u(x,t)$, we conclude that $u(x,t)\equiv0$.
\end{proof}

\section{\label{sect6}Markov Processes}

In this section we show that the fundamental solution ${\LARGE \Lambda
}(x,t,\xi,\tau)$ of Cauchy problem (\ref{Cauchy2}) is the transition density
of a Markov process. We need some preliminary results.

\begin{lemma}
\label{lemma1_mp}If the coefficients $a_{k}(x,t)$ and $b(x,t)$ are
nonnegative, then%
\[
{\LARGE \Lambda}(x,t,\xi,\tau)\geq0.
\]

\end{lemma}

\begin{proof}
It is sufficient to show that%
\[
u(x,t)=%
{\displaystyle\int\limits_{\mathbb{Q}_{p}^{n}}}
{\LARGE \Lambda}(x,t,\xi,\tau)\varphi(\xi)d^{n}\xi\geq0,
\]
where $u(x,t)$ is the solution of Cauchy problem (\ref{Cauchy2}) with
$f(x,t)\equiv0$, and initial condition $u(x,0)=\varphi(x)\geq0$ with$\ \varphi
\in S(\mathbb{Q}_{p}^{n})$. From (\ref{Gamma}), (\ref{Cotateo}), and Lemma
\ref{lemmaZ} (iii), it follows that
\begin{equation}
u(x,t)\rightarrow0\text{ as }\left\Vert x\right\Vert _{p}\rightarrow\infty.
\label{Eqlim}%
\end{equation}

Now, if $u(x,t)<0$, then there exist $x_{0}\in\mathbb{Q}_{p}^{n}$ and
$t_{0}\in(0,T]$ such that%
\begin{equation}
\inf_{\substack{0\leq t\leq T,\\x\in\mathbb{Q}_{p}^{n}}}u\left(  x,t\right)
=u(x_{0},t_{0})<0. \label{Eqinf}%
\end{equation}
This implies that $(\mathbf{W}_{\alpha}u)(x_{0},t_{0})\geq0$, $(\mathbf{W}%
_{\alpha_{k}}u)(x_{0},t_{0})\geq0$ for all $k$,\ and $\frac{\partial
u}{\partial t}(x_{0},t_{0})\leq0$. On the other hand,
\[
\frac{\partial u}{\partial t}(x,t)-a_{0}(x,t)(\boldsymbol{\mathbf{W}}_{\alpha
}u)(x,t)-\sum_{k=1}^{N}a_{k}(x,t)(\boldsymbol{\mathbf{W}}_{\alpha_{k}%
}u)(x,t)=0.
\]
By using the uniform parabolicity condition $a_{0}(x,t)\geq\mu>0$, we get
$(\mathbf{W}_{\alpha}u)(x_{0},t_{0})$ $=0$,\ then by (\ref{W}), $u(x,t_{0}%
)$\ is constant, and by (\ref{Eqlim}), $u(x,t_{0})\equiv0$, which contradicts
(\ref{Eqinf}).
\end{proof}

\begin{lemma}
\label{lemma2_mp}If $b(x,t)\equiv0$, then%
\[%
{\displaystyle\int\limits_{\mathbb{Q}_{p}^{n}}}
{\LARGE \Lambda}(x,t,\xi,\tau)d^{n}\xi=1.
\]

\end{lemma}

\begin{proof}
By integrating (\ref{Gamma1}) in the variable $\xi$ over whole the
space\ $\mathbb{Q}_{p}^{n},$ and by using Lemma \ref{lemmaWZ} (iii), we have%
\[
\frac{\partial}{\partial t}\left(
{\displaystyle\int\limits_{\mathbb{Q}_{p}^{n}}}
{\LARGE \Lambda}(x,t,\xi,\tau)d^{n}\xi\right)  =0,
\]
thus $\int_{\mathbb{Q}_{p}^{n}}{\LARGE \Lambda}(x,t,\xi,\tau)d^{n}\xi$ is
independent of $t$. Now, by integrating (\ref{Gamma}) over whole space
$\mathbb{Q}_{p}^{n}$ in variable $\xi$ and by using Lemma \ref{lemmaZ} (iv),
we have%
\[%
{\displaystyle\int\limits_{\mathbb{Q}_{p}^{n}}}
{\LARGE \Lambda}(x,t,\xi,\tau)d^{n}\xi=1+%
{\displaystyle\int\limits_{\tau}^{t}}
\underset{\mathbb{Q}_{p}^{n}}{\int}\underset{\mathbb{Q}_{p}^{n}}{\int}%
Z(x-\eta,t-\theta,\eta,\theta)\phi(\eta,\theta,\xi,\tau)d^{n}\eta d^{n}\xi
d\theta.
\]
The result is obtained by taking $t=\tau$ in the above formula.
\end{proof}

\begin{lemma}
\label{lemma3_mp}If $b(x,t)\equiv0$ and $f(x,t)\equiv0$, then the function
${\LARGE \Lambda}(x,t,\xi,\tau)$ satisfies the following property:%
\begin{equation}
{\LARGE \Lambda}(x,t,\xi,\tau)=%
{\displaystyle\int\limits_{\mathbb{Q}_{p}^{n}}}
{\LARGE \Lambda}(x,t,y,\sigma){\LARGE \Lambda}(y,\sigma,\xi,\tau)d^{n}y.
\label{IdentityL}%
\end{equation}

\end{lemma}

\begin{proof}
Consider the following initial value problem:%
\begin{equation}
\left\{
\begin{array}
[c]{l}%
\frac{\partial u}{\partial t}(x,t)-a_{0}(x,t)(\boldsymbol{\mathbf{W}}_{\alpha
}u)(x,t)-(\widetilde{\boldsymbol{W}}u)(x,t)=0\\
\\
u\left(  x,\tau\right)  =\varphi(x),\text{\ }x\in\mathbb{Q}_{p}^{n}\text{ and
}t\in(\tau,\sigma],
\end{array}
\right.  \label{Cauchy2.5}%
\end{equation}
by Theorem \ref{mainTheo}, $u(x,\sigma)=\int_{\mathbb{Q}_{p}^{n}%
}{\LARGE \Lambda}(x,\sigma,\xi,\tau)\varphi(\xi)d^{n}\xi$. Now consider%
\begin{equation}
\left\{
\begin{array}
[c]{l}%
\frac{\partial u}{\partial t}(x,t)-a_{0}(x,t)(\boldsymbol{\mathbf{W}}_{\alpha
}u)(x,t)-(\widetilde{\boldsymbol{W}}u)(x,t)=0\\
\\
u\left(  x,\sigma\right)  =%
{\displaystyle\int\limits_{\mathbb{Q}_{p}^{n}}}
{\LARGE \Lambda}(x,\sigma,\xi,\tau)\varphi(\xi)d^{n}\xi,\text{\ }%
x\in\mathbb{Q}_{p}^{n}\text{, }t\in(\sigma,T],\text{ with \ }\tau<\sigma<T,
\end{array}
\right.  \label{Cauchy2.6}%
\end{equation}
by Theorem \ref{mainTheo} and Fubini's Theorem, the solution of
(\ref{Cauchy2.6}) is given by%
\[
u(x,t)=%
{\displaystyle\int\limits_{\mathbb{Q}_{p}^{n}}}
\left(
{\displaystyle\int\limits_{\mathbb{Q}_{p}^{n}}}
{\LARGE \Lambda}(x,t,y,\sigma){\LARGE \Lambda}(y,\sigma,\xi,\tau
)d^{n}y\right)  \varphi(\xi)d^{n}\xi.
\]
On the other hand, (\ref{Cauchy2.6}) is equivalent to%
\begin{equation}
\left\{
\begin{array}
[c]{l}%
\frac{\partial u}{\partial t}(x,t)-a_{0}(x,t)(\boldsymbol{\mathbf{W}}_{\alpha
}u)(x,t)-(\widetilde{\boldsymbol{W}}u)(x,t)=0\\
\\
u\left(  x,\tau\right)  =\varphi(x),\text{ \ }x\in\mathbb{Q}_{p}^{n}\text{
},\text{ }t\in\left(  \tau,T\right]  ,
\end{array}
\right.  \label{Cauchy2.7}%
\end{equation}
which has solution given by $u\left(  x,t\right)  =\int_{\mathbb{Q}_{p}^{n}%
}{\LARGE \Lambda}(x,t,\xi,\tau)\varphi(\xi)d^{n}\xi$. Now, by Theorem
\ref{uniqueTheo},%
\[%
{\displaystyle\int\limits_{\mathbb{Q}_{p}^{n}}}
{\LARGE \Lambda}(x,t,\xi,\tau)\varphi(\xi)d^{n}\xi=%
{\displaystyle\int\limits_{\mathbb{Q}_{p}^{n}}}
\left(
{\displaystyle\int\limits_{\mathbb{Q}_{p}^{n}}}
{\LARGE \Lambda}(x,t,y,\sigma){\LARGE \Lambda}(y,\sigma,\xi,\tau
)d^{n}y\right)  \varphi(\xi)d^{n}\xi,
\]
for any test function $\varphi$, which implies (\ref{IdentityL}).
\end{proof}

\begin{theorem}
\label{Thm2}If the coefficients $a_{k}(x,t),$ $k=1,\cdots,N$ are nonnegative
bounded continuous functions, $b(x,t)\equiv0$, $0\leq\lambda<\alpha_{1}-n$ (if
$a_{1}(x,t)=\cdots$ $a_{k}(x,t)\equiv0,$ we shall \ suppose that $0\leq
\lambda<\alpha-n$), and $f(x,t)\equiv0$, then the fundamental solution
${\LARGE \Lambda}(x,t,\xi,\tau)$ is the transition density of a bounded
right-continuous Markov process without second kind discontinuities.
\end{theorem}

\begin{proof}
The result follows from \cite[Theorem 3.6]{Dyn} by using Lemmas
(\ref{lemma1_mp})-(\ref{lemma2_mp})-(\ref{lemma3_mp}), and (\ref{Gamma}%
)-(\ref{Cotateo}), and Lemma \ref{lemmaZ} (iii).
\end{proof}

\section{The Cauchy Problem is Well-Posed}

In this section, we study the continuity of the solution of Cauchy problem
(\ref{Cauchy2}) with respect to $\varphi\left(  x\right)  $\ and $f\left(
x,t\right)  $. We assume that the coefficients $a_{k}(x,t),$ $k=0,1,\ldots,N$
are nonnegative bounded continuous functions, $b(x,t)$ is a bounded continuous
function, $0\leq\lambda<\alpha_{1}-n$ (if $a_{1}(x,t)=\ldots=a_{k}%
(x,t)\equiv0$, we shall \ suppose that $0\leq\lambda<\alpha-n$),
$\varphi\left(  x\right)  \in\mathfrak{M}_{\lambda}$\ and $f\left(
x,t\right)  \in\mathfrak{M}_{\lambda}$, uniformly in $t$, with \ $0\leq
\lambda<\alpha_{1}-n$.

We identify $\mathfrak{M}_{\lambda}$ with the $\mathbb{R}$-vector space of all
the functions \textquotedblleft$\phi\left(  x,t\right)  \in\mathfrak{M}%
_{\lambda}$, uniformly in $t$,\textquotedblright\ and introduce on
$\mathfrak{M}_{\lambda}$ the following norm:%
\[
\left\Vert \phi\right\Vert _{\mathfrak{M}_{\lambda}}:=\sup_{t\in\left[
0,T\right]  }\sup_{x\in\mathbb{Q}_{p}^{n}}\left\vert \frac{\phi\left(
x,t\right)  }{1+\left\Vert x\right\Vert _{p}^{\lambda}}\right\vert .
\]
From now on, we consider $\mathfrak{M}_{\lambda}$ as \ topological vector
space with the topology induced by $\left\Vert \cdot\right\Vert _{\mathfrak{M}%
_{\lambda}}$. We also consider $\mathfrak{M}_{\lambda}\times\mathfrak{M}%
_{\lambda}$\ as topological vector space with the topology induced by the norm
$\left\Vert \cdot\right\Vert _{\mathfrak{M}_{\lambda}}+\left\Vert
\star\right\Vert _{\mathfrak{M}_{\lambda}}$.

\begin{theorem}
\label{Thm3}With \ the above hypotheses, consider the following operator:%
\[%
\begin{array}
[c]{ccc}%
\mathfrak{M}_{\lambda}\times\mathfrak{M}_{\lambda} & \underrightarrow
{\boldsymbol{L}} & \mathfrak{M}_{\lambda}\\
&  & \\
\left(  \varphi\left(  x\right)  ,f\left(  x,t\right)  \right)  & \rightarrow
& u\left(  x,t\right)  ,
\end{array}
\]
where $u\left(  x,t\right)  $ is given by (\ref{u(x,t)}). Then $\left\Vert
u\left(  x,t\right)  \right\Vert _{\mathfrak{M}_{\lambda}}\leq C\left(
\left\Vert \varphi\left(  x\right)  \right\Vert _{\mathfrak{M}_{\lambda}%
}+\left\Vert f\left(  x,t\right)  \right\Vert _{\mathfrak{M}_{\lambda}%
}\right)  $, i.e. $\boldsymbol{L}$ is a continuous operator.
\end{theorem}

\begin{proof}
We write $u(x,t)=u_{1}(x,t)+u_{2}(x,t)$ where%
\[
u_{1}(x,t)=%
{\displaystyle\int\limits_{0}^{t}}
{\displaystyle\int\limits_{\mathbb{Q}_{p}^{n}}}
{\LARGE \Lambda}(x,t,\xi,\tau)f(\xi,\tau)d^{n}\xi d\tau\text{ \ and\ }%
u_{2}(x,t)=%
{\displaystyle\int\limits_{\mathbb{Q}_{p}^{n}}}
{\LARGE \Lambda}(x,t,\xi,0)\varphi(\xi)d^{n}\xi
\]
as before. Now%
\begin{multline*}
\left\vert u_{1}(x,t)\right\vert \leq%
{\displaystyle\int\limits_{0}^{t}}
{\displaystyle\int\limits_{\mathbb{Q}_{p}^{n}}}
\left\vert {\LARGE \Lambda}(x,t,\xi,\tau)\right\vert \left\vert f(\xi
,\tau)\right\vert d^{n}\xi d\tau\\
\leq\left\Vert f\left(  x,t\right)  \right\Vert _{\mathfrak{M}_{\lambda}%
}\left\{
{\displaystyle\int\limits_{0}^{t}}
{\displaystyle\int\limits_{\mathbb{Q}_{p}^{n}}}
\left\vert {\LARGE \Lambda}(x,t,\xi,\tau)\right\vert d^{n}\xi d\tau+%
{\displaystyle\int\limits_{0}^{t}}
{\displaystyle\int\limits_{\mathbb{Q}_{p}^{n}}}
\left\vert {\LARGE \Lambda}(x,t,\xi,\tau)\right\vert \left\Vert \xi\right\Vert
_{p}^{\lambda}d^{n}\xi d\tau\right\}  ,
\end{multline*}
by (\ref{Gamma})-(\ref{Cotateo}), (\ref{Est1}) and Proposition \ref{Conv},%
\begin{multline*}
\left\vert u_{1}(x,t)\right\vert \leq C_{0}\left\Vert f\left(  x,t\right)
\right\Vert _{\mathfrak{M}_{\lambda}}\left\{
{\displaystyle\int\limits_{0}^{t}}
\left(  t-\tau\right)  ^{1+\frac{n-\alpha}{\alpha-n}}d\tau+%
{\displaystyle\int\limits_{0}^{t}}
\left(  t-\tau\right)  ^{2-\frac{\lambda}{\alpha-n}+\frac{n-\alpha}{\alpha-n}%
}d\tau\right. \\
+%
{\displaystyle\sum\limits_{k=1}^{N+1}}
{\displaystyle\int\limits_{0}^{t}}
\left(  t-\tau\right)  ^{1+\frac{n-\alpha_{k}}{\alpha-n}}d\tau+\left(
1+\left\Vert x\right\Vert _{p}^{\lambda}\right)
{\displaystyle\int\limits_{0}^{t}}
\left(  t-\tau\right)  ^{1+\frac{n-\alpha}{\alpha-n}}d\tau\\
+\left(  1+\left\Vert x\right\Vert _{p}^{\lambda}\right)
{\displaystyle\int\limits_{0}^{t}}
\left(  t-\tau\right)  ^{2-\frac{\lambda}{\alpha-n}+\frac{n-\alpha}{\alpha-n}%
}d\tau+\left(  1+\left\Vert x\right\Vert _{p}^{\lambda}\right)  \times\\
\left.
{\displaystyle\sum\limits_{k=1}^{N+1}}
{\displaystyle\int\limits_{0}^{t}}
\left(  t-\tau\right)  ^{1+\frac{n-\alpha_{k}}{\alpha-n}}d\tau\right\}
\leq\left\Vert f\left(  x,t\right)  \right\Vert _{\mathfrak{M}_{\lambda}%
}\left\{  C_{1}\left(  T\right)  +C_{2}\left(  T\right)  \left(  1+\left\Vert
x\right\Vert _{p}^{\lambda}\right)  \right\}  .
\end{multline*}
Hence,\
\[
\left\vert \frac{u_{1}(x,t)}{1+\left\Vert x\right\Vert _{p}^{\lambda}%
}\right\vert \leq\left\Vert f\left(  x,t\right)  \right\Vert _{\mathfrak{M}%
_{\lambda}}\left\{  \frac{C_{1}\left(  T\right)  }{1+\left\Vert x\right\Vert
_{p}^{\lambda}}+C_{2}\left(  T\right)  \right\}  .
\]

In the same form, one shows that
\[
\left\vert \frac{u_{2}(x,t)}{1+\left\Vert x\right\Vert _{p}^{\lambda}%
}\right\vert \leq\left\Vert \varphi\left(  x\right)  \right\Vert
_{\mathfrak{M}_{\lambda}}\left\{  \frac{C_{1}^{\prime}\left(  T\right)
}{1+\left\Vert x\right\Vert _{p}^{\lambda}}+C_{2}^{\prime}\left(  T\right)
\right\}  ,
\]
therefore $\left\Vert u\left(  x,t\right)  \right\Vert _{\mathfrak{M}%
_{\lambda}}\leq C\left(  \left\Vert \varphi\left(  x,t\right)  \right\Vert
_{\mathfrak{M}_{\lambda}}+\left\Vert f\left(  x,t\right)  \right\Vert
_{\mathfrak{M}_{\lambda}}\right)  $.
\end{proof}

\begin{acknowledgement}
The authors wish to thank to Sergii Torba for many useful comments and
discussions, which led to an improvement of this work.\bigskip
\end{acknowledgement}


\begin{thebibliography}{99}                                                                                               %


\bibitem {A-K-S}Albeverio S., Khrennikov A. Yu., Shelkovich V. M., Theory of
$p$-adic distributions: linear and nonlinear models. Cambridge University
Press, 2010.

\bibitem {A-K1}Albeverio S., Karwoski W., Diffusion in $p-$adic numbers. In:
K. Ito, H. Hida (Eds.),\ Gaussian Random Fields, pp. 86-99, 1991, World
Scientific, Singapore.

\bibitem {A-K2}Albeverio S., Karwoski W., A random walk on $p-$adics: the
generator and its spectrum, Stochastic Process. Appl. 53 (1994), 1-22.

\bibitem {Av-3}Avetisov V. A., Bikulov A. Kh., Osipov V. A., $p$-adic models
of ultrametric diffusion in the conformational dynamics of macromolecules,
Proc. Steklov Inst. Math. 2004, no. 2 (245), 48--57.

\bibitem {Av-4}Avetisov V. A., Bikulov A. Kh., Osipov V. A., $p$-adic
description of characteristic relaxation in complex systems, J. Phys. A 36
(2003), no. 15, 4239--4246.

\bibitem {Av-5}Avetisov V. A., Bikulov A. H., Kozyrev S. V., Osipov V. A.,
$p$-adic models of ultrametric diffusion constrained by hierarchical energy
landscapes, J. Phys. A 35 (2002), no. 2, 177--189.

\bibitem {C-Z2}Casas-Sánchez O.F., Zúñiga-Galindo W. A.,$\ p$-adic elliptic
quadratic forms, parabolic-type pseudodifferential equations with variable
coefficients and Markov processes, $p$-Adic Numbers Ultrametric Anal. Appl. 6
(2014), no. 1, 1--20.

\bibitem {Ch-Z}Chacón-Cortes L.F., Zúñiga-Galindo W. A., Nonlocal Operators,
Parabolic-Type Equations, and Ultrametric Random Walks, J. Math. Phys. 54,
113503 (2013).

\bibitem {Dra-Kh-K-V}Dragovich B., Khrennikov A. Yu., Kozyrev S. V., Volovich,
I. V., On $p$-adic mathematical physics, p-Adic Numbers Ultrametric Anal.
Appl. 1 (2009), no. 1, 1--17.

\bibitem {Dyn}Dynkin E. B., Markov processes. Vol. I. Springer-Verlag, 1965.

\bibitem {Ka}Karwowski W., Diffusion processes with ultrametric jumps, Rep.
Math. Phys. 60 (2007), no. 2, 221--235.

\bibitem {Koch0}Kochubei A.N. , Parabolic equations over the field of $p$-adic
numbers, Math. USSR-Izv. 39 (1992), no. 3, 1263--1280.

\bibitem {Koch}Kochubei Anatoly N., Pseudo-differential equations and
stochastics over non-Archimedean fields. Marcel Dekker, Inc., New York, 2001.

\bibitem {K-Kos}Khrennikov A. Yu., Kozyrev S. V., $p$-adic pseudodifferential
operators and analytic continuation of replica matrices, Theoret. and Math.
Phys. 144 (2005), no. 2, 1166--1170.

\bibitem {M-P-V}Mézard Marc, Parisi Giorgio, Virasoro Miguel Angel, Spin glass
theory and beyond. World Scientific, 1987.

\bibitem {R-T-V}Rammal R., Toulouse G., Virasoro M. A., Ultrametricity for
physicists, Rev. Modern Phys. 58 (1986), no. 3, 765--788.

\bibitem {Ro}Rodríguez-Vega John Jaime, On a general type of $p$-adic
parabolic equations, Rev. Colombiana Mat. 43 (2009), no. 2, 101--114.

\bibitem {R-Zu}Rodríguez-Vega J. J., Zúñiga-Galindo W. A., Taibleson
operators, $p$-adic parabolic equations and ultrametric diffusion, Pacific J.
Math. 237 (2008), no. 2, 327--347.

\bibitem {Taibleson}Taibleson M. H., Fourier analysis on local fields,
Princeton University Press, 1975.

\bibitem {T-Z}Torba S. M., Zúñiga-Galindo W. A., Parabolic Type Equations and
Markov Stochastic Processes on Adeles, J. Fourier Anal. Appl. 19 (2013), no.
4, 792835.

\bibitem {Va1}Varadarajan V. S., Path integrals for a class of $p$-adic
Schrödinger equations, Lett. Math. Phys. 39 (1997), no. 2, 97--106.

\bibitem {V-V-Z}Vladimirov V. S., Volovich I. V., Zelenov E. I., $p$-adic
analysis and mathematical physics, World Scientific, 1994.

\bibitem {Zu}Zúñiga-Galindo W. A., Parabolic equations and Markov processes
over p-adic fields, Potential Anal. 28 (2008), no. 2, 185--200.
\end{thebibliography}
\end{document}